\documentclass[12pt]{amsart}

\usepackage[margin=1in,letterpaper,portrait]{geometry}
\usepackage{amsmath}
\usepackage{amssymb}
\usepackage{amsthm}
\usepackage{mathrsfs}
\usepackage{verbatim}
\usepackage{fancyhdr}
\usepackage{url}
\usepackage{tikz}
\usepackage{enumerate}
\usepackage{longtable}
\usetikzlibrary{matrix,arrows}
\usepackage{tikz-cd}
\usepackage[pdftex,bookmarksnumbered=true,linkbordercolor={1 0 0}]{hyperref}
\usetikzlibrary{plotmarks}
\usepackage{pgfplots}
\usepgfplotslibrary{fillbetween}

\usepackage{setspace}

\newcommand{\student}{Dennis Tseng}
\newcommand{\studentemail}{DennisCTseng@gmail.com}

\newcommand{\mb}{\mathbb}

\newcommand{\mc}{\mathcal}

\newcommand{\wt}{\widetilde}
\newcommand{\codim}{{\rm codim}}
\newcommand{\Hilb}{{\rm Hilb}}
\newcommand{\RHilb}{\wt{{\rm Hilb}}}

\newcommand{\Spec}{{\rm Spec}}
\newcommand{\Prob}{{\rm Prob}}
\newcommand{\F}{F_{1}}
\newcommand{\Fo}{F_{0,1}}
\newcommand{\Co}{{\rm Contain}}
\newcommand{\Span}{{\rm Span}}
\newcommand{\Hyp}{{\rm Hyp}}

\newcommand*{\sheafhom}{\mathscr{H}\kern -.5pt om}

\begin{document}
\theoremstyle{plain}
\newtheorem{Thm}{Theorem}[section]
\newtheorem{Cor}[Thm]{Corollary}
\newtheorem{Conj}[Thm]{Conjecture}
\newtheorem{Pro}[Thm]{Problem}
\newtheorem{Main}{Main Theorem}
\renewcommand{\theMain}{}
\newcommand{\Sol}{\text{Sol}}
\newtheorem{Lem}[Thm]{Lemma}
\newtheorem{Claim}[Thm]{Claim}
\newtheorem{Prop}[Thm]{Proposition}
\newtheorem{Exam}{Example}
\newtheorem{Rem}{Remark}
\newtheorem{ToDo}{To Do}

\theoremstyle{definition}
\newtheorem{Def}{Definition}[section]
\newtheorem{Exer}{Exercise}

\theoremstyle{remark}

\title{Collections of Hypersurfaces Containing a Curve}
\author{Dennis Tseng}
\address{Dennis Tseng, Harvard University, Cambridge, MA 02138}
\email{DennisCTseng@gmail.com}
\date{\today}

\allowdisplaybreaks

\begin{abstract}
We consider the closed locus parameterizing $k$-tuples of hypersurfaces that have positive dimensional intersection and fail to intersect properly, and show in a large range of degrees that its unique irreducible component of maximal dimension consists of tuples of hypersurfaces whose intersection contains a line. We then apply our methods in conjunction with a known reduction to positive characteristic argument to find the unique component of maximal dimension of the locus of hypersurfaces with positive dimensional singular loci. We will also find the components of maximal dimension of the locus of smooth hypersurfaces with a higher dimensional family of lines through a point than expected.
\end{abstract}

\maketitle

\section{Introduction}
A general choice of $k$ hypersurfaces in $\mb{P}^r$ will intersect in a locus of dimension $r-k$. Equivalently, a general choice of $k$ homogenous polynomials in $K[X_0,\ldots,X_r]$ form a regular sequence. It is a closed condition for a sequence of hypersurfaces to not intersect properly, and the purpose of this paper is to address basic questions about this locus. 

In the second part of the paper, we will give two applications that are of independent interest. Specifically, we will look at the locus of hypersurfaces with positive dimensional singular locus and the locus of hypersurfaces with more lines through a point than expected, improving on previously known results. 

\subsection{Problem statement}
A first natural question would be the following:
\begin{Pro}
\label{PROBA}
Let $Z\subset \prod_{i=1}^{k}\mb{A}^{\binom{r+d_i}{d_i}}$ parameterize the $k$-tuples $(F_1,\ldots,F_{k})$ of homogenous forms of degrees $d_1,\ldots,d_k$ where $\dim(V(F_1,\ldots,F_k))$ exceeds $r-k$. What is the dimension of $Z$, and what are its components of maximal dimension?
\end{Pro}

For example, one might naively expect that the largest component of $Z$ is the locus of $k$-tuples of forms all containing the same linear space of dimension $r-k+1$, because linear spaces are the simplest $r-k+1$ dimensional subvarieties of projective space.
\begin{Pro}
\label{PROBB}
Does $Z$ have a unique component of maximal dimension, consisting of tuples $(F_1,\ldots,F_k)$ of hypersurfaces all containing the same $r-k+1$ dimensional linear space?
\end{Pro}

The answer to Problem \ref{PROBB} is negative as it stands. For example, if $r=3$ and the degrees are $d_1=2$, $d_2=2$, and $d_3=100$, then the locus of 3-tuples of hypersurfaces all containing the same line is codimension 103, while the second quadric will be equal to the first quadric in codimension 9. Even if the degrees are all equal, we can let $r=4$, $k=2$, and the degrees be $d_1=2$, $d_2=2$, where the two quadrics will contain a plane in codimension 16, but are equal in codimension 14. 

\subsection{Results and applications}
In spite of the simple counterexamples above, we will show that Problem \ref{PROBB} has a positive answer in many cases when $k=r$, as stated in Theorem \ref{NR}. 
\begin{Thm}
\label{NR}
If $2\leq d_1\leq d_2\leq \cdots\leq d_r$ and $d_i\leq d_1+\binom{d_1}{2}(i-1)$, then the locus $Z\subset \prod_{i=1}^{r}\mb{A}^{\binom{r+d_i}{d_i}}$ has codimension
\begin{align*}
-2(r-1)+\sum_{i=1}^{r}{(d_i+1)}.
\end{align*}
Furthermore, the unique component of maximal dimension consists of $r$-tuples of hypersurfaces all containing the same line. 
\end{Thm}
Figure \ref{F1} depicts the restriction on the degrees in Theorem \ref{NR}. 

In fact, it is no harder to prove the analogous result in the more general case where $Z$ is the locus of $k$-tuples of hypersurfaces having positive dimensional intersection for $k\geq r$.  This is stated in Theorem \ref{slope}, proving the result claimed in the abstract. 
\begin{figure}[h]
\begin{tikzpicture} 
    \begin{axis}[xmin=-.5, xmax=10.5,
   ymin=0, ymax= 35,      yticklabels={,,},
      xticklabels={,,}, ticks=none]
    \addplot[name path=f,domain=0:10,black, dashed] {3*x+3};
    \addplot[name path=g,domain=0:10,black, dashed] {3};

    \path[name path=axis] (axis cs:0,0) -- (axis cs:10,0);

    \addplot [
        thick,
        color=black,
        fill=black, 
        fill opacity=0.05
    ]
    fill between[
        of=f and g,
        soft clip={domain=0:10},
    ];

    \node [rotate=43] at (axis cs:  5,  21) {Slope $\binom{d_1}{2}$};
\node  at (axis cs:  0,  1.5) {$d_1$};
\node  at (axis cs:  1,  2.5) {$d_2$};
\node  at (axis cs:  10,  25.5) {$d_r$};

\addplot[mark=none, domain=0:1,black]{
x+3
};
\addplot[mark=none, domain=1:2,black]{
3*(x-1)+4
};
\addplot[mark=none, domain=2:5,black]{
7
};
\addplot[mark=none, domain=5:6,black]{
14*(x-5)+7
};
\addplot[mark=none, domain=6:8,black]{
1*(x-6)+21
};
\addplot[mark=none, domain=8:10,black]{
2*(x-8)+23
};
    \end{axis}
\end{tikzpicture}
\caption{Allowable range of degrees in Theorem \ref{NR}}
\label{F1}
\end{figure}
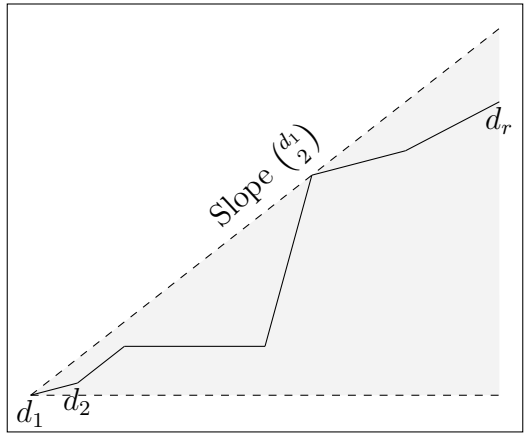

We give two applications. 
\subsubsection{Singular hypersurfaces}
Consider the following conjecture studied by Slavov \cite{Kaloyan}:
\begin{Conj}
\label{CONJA}
Among hypersurfaces of degree $\ell$ in $\mb{P}^r$ with singular locus of dimension at least $b$, the unique component of maximal dimension consists of hypersurfaces singular along a $b$-dimensional linear space. 
\end{Conj}

The obstruction to solving Conjecture \ref{CONJA} is similar to the obstruction to solving Problem \ref{PROBA} outlined above. One would expect that it is easier to be singular along a linear space than a more complicated variety, but a more complicated variety might vary in a family with more moduli. 

As progress towards Conjecture \ref{CONJA}, Slavov showed:
\begin{Thm} [{\cite[Theorem 1.1]{Kaloyan}}]
\label{KaloyanT}
Given $r, b, p$, there exists an effectively computable integer $\ell_0=\ell_0(r,b,p)$ such that for all $\ell\geq \ell_0$, Conjecture \ref{CONJA} holds over a field of characteristic $p$, where $p$ is possibly 0.
\end{Thm}

We can give a much better bound for $\ell_0$ in characteristic 0 when $b=1$. We will show
\begin{Thm}
\label{TsengT}
In characteristic 0 or 2, if $\ell\geq 7$ or $\ell=5$, among hypersurfaces with positive dimensional singular locus, the unique component of maximal dimension is the locus of hypersurfaces singular along a line. 
\end{Thm}

We will prove Theorem \ref{TsengT} by applying the trick first given by Poonen \cite{Poonen} to decouple the partial derivatives of a homogenous form in positive characteristic to prove Theorem \ref{KaloyanT}.  Specifically, we will use Slavov's argument \cite{Kaloyan} to reduce the problem to understanding the locus of  $(r+1)$-tuples $(F_1,\ldots,F_{r+1})$ of hypersurfaces in $\mb{P}^r$ of the same degree, where $V(F_1, \cdots,F_{r+1})$ is positive dimensional, in characteristic 2.

\subsubsection{Lines on hypersurfaces}
The second application concerns lines on hypersurfaces, in particular, smooth hypersurfaces with unexpectedly large dimensions of lines through a given point. These points are called \emph{1-sharp} points in \cite{RY}. We are able to find the maximal dimensional components of the space of hypersurfaces with a 1-sharp point in \emph{all} dimensions and degrees. 

We would expect an $r-d-1$ dimensional family of lines through a point on a smooth hypersurface of degree $d$ in $\mb{P}^r$. The case $d\geq r$ is trivial, since we are not expecting any lines through a given point, the case $d\leq r-2$ can be deduced from the crude degeneration given in the proof of \cite[Theorem 2.1]{LowDegree}. Therefore, the last open case is $d=r-1$, and we show:

\begin{Thm}
\label{EVMI}
Let $U\subset \mb{P}^{\binom{2r-1}{r-1}-1}$ be the open set of smooth hypersurfaces of degree $r-1\geq 3$ in $\mb{P}^r$. Let $Z\subset U$ be the closed subset of hypersurfaces $X$ that contain a positive dimensional family of lines through a point. Then, there are two components of $Z$ of maximal dimension for $r=5$ and a unique component of maximal dimension for $r\neq  5$. The components of maximal dimensions are:
\begin{enumerate}
\item the locus of hypersurfaces where the second fundamental form vanishes at a point for $r\leq 5$
\item the locus of hypersurfaces containing a 2-plane for $r\geq 5$.
\end{enumerate}
\end{Thm}
For a more detailed description of what happens in all degrees and dimensions, see Theorem \ref{EVM} below. 

Lines through points on hypersurfaces are relevant to the following conjecture by Coskun, Harris, and Starr: 
\begin{Conj}
[{\cite[Conjecture 1.3]{cubic}}]
\label{CONJB}
Let $X\subset\mb{P}^r$ be a general Fano hypersurface of degree at least 3 and $R_e(X)$ denote the closure of the locus of the Hilbert scheme parameterizing smooth degree $e$ rational curves on $X$. Then, $R_e(X)$ is irreducible of dimension $e(r+1-d)+r-4$. 
\end{Conj}

For work on Conjecture \ref{CONJB}, see \cite{LowDegree, Beheshti, RY}. For the related problem regarding rational curves on a arbitrary smooth hypersurface of very low degree see \cite{cubic, TDBrowning}. 


The best-known bound to Conjecture \ref{CONJB} was given by Riedl-Yang \cite{RY}, where the authors proved Conjecture \ref{CONJB} for $d\leq r-2$ by applying bend and break to rational curves within a complete family of hypersurfaces to reduce to the case of lines, so it was crucial for them to know the codimension of the locus of hypersurfaces with more lines through points than expected. We hope that Theorem \ref{EVMI} will be useful in proving the $d=r-1$ case of Conjecture \ref{CONJB}. An attempt at using Theorem \ref{EVMI} to prove the $d=r-1$ case that works only for when $e$ is at most roughly $\frac{2-\sqrt{2}}{2}n$ is given in \cite{Tseng2}. 

\subsection{Previous work}
The naive incidence correspondence between $k$-tuples of hypersurfaces and possible $r-k+1$ -dimensional schemes contained in the intersection quickly runs into issues regarding the dimension of the Hilbert scheme and bounds for the Hilbert function, neither of which are well understood. 

To the author's knowledge, the only known method in the literature to approach Problem \ref{PROBA} is through a crude degeneration. Namely, in order to bound the locus where $(F_1,\ldots,F_i)$ form a complete intersection but $F_{i+1}$ contains a component of $V(F_1, \cdots, F_i)$, we can linearly degenerate each component of $V(F_1, \cdots,F_i)$ to lie in a linear space and forget the scheme-theoretic structure. This method was used in \cite[Theorem 2.1]{LowDegree}. 

Similarly, we can linearly degenerate each component of $V(F_1, \cdots, F_i)$ to a union of linear spaces \cite[Lemma 4.5]{Kaloyan}. However, applying this requires a separate incidence correspondence to deal with the case where $V(F_1, \cdots, F_k)$ contains a $r-k+1$ dimensional variety of low degree, making it difficult to obtain quantitative bounds. 

To illustrate the limitations of the known methods, the solution to Problem \ref{PROBA} was unknown to the author even in the case $k=r$ and $d_1=\cdots=d_r=2$, where we are intersecting $r$ quadrics in $\mb{P}^r$. 

\subsubsection{A qualitative answer}
We note that Problem \ref{PROBB} also has an affirmative answer after a large enough twist.
\begin{Thm} [{\cite[Corollary 1.3]{Tseng}}]
\label{HYP}
Given degrees $d_1,\ldots,d_k$, there exists $N_0$ such that for $N\geq N_0$, the unique component of maximal dimension of the locus 
\begin{align*}
Z\subset \prod_{i=1}^{k}\mb{A}^{\binom{r+d_i+N}{d_i+N}}
\end{align*}
of $k$-tuples of hypersurfaces of degrees $d_1+N,\ldots,d_k+N$ that fail to intersect properly is the locus where the hypersurfaces all contain the same codimension $k-1$ linear space.
\end{Thm}
In fact, the author has shown that Theorem \ref{HYP} holds in much greater generality, where we replace $\mb{P}^r$ with an arbitrary variety $X$ and replace the choice of $k$ hypersurfaces by a section of a vector bundle $V$ \cite[Theorem 1.2]{Tseng}. However, the key difference is that the method in this paper gives effective bounds, which is required for the applications. As an analogy, both Theorems \ref{HYP} and \ref{KaloyanT} require a large twist, and Theorems \ref{NR} and \ref{TsengT} are effective versions. 

\subsection{Proof idea}
\label{PI}
The proof is elementary, and the only input is a lower bound for the Hilbert function of a nondegenerate variety. This is stated in \cite[Theorem 1.3]{Park}, for example. 
\begin{Lem}
\label{ARBH}
If $X\subset \mb{P}^r$ is a nondegenerate variety of dimension $a$, then the Hilbert function $h_X$ of $X$ is bounded below by
\begin{align*}
h_{X}(d) \geq (r-a)\binom{d+a-1}{d-1}+\binom{d+a}{d}.
\end{align*}
\end{Lem}
Note that Lemma \ref{ARBH} is also elementary, as it can be proven by repeatedly cutting it with a general hyperplane and applying \cite[Lemma 3.1]{Montreal}. The varieties $X$ for which equality is satisfied are the varieties of minimal degree \cite[Remark 1.6 (1)]{Park}. 

To see the proof worked out completely in an example, see Section \ref{KEX}. In this section, we give a feeling for the argument by briefly describing how it might arise naturally. Suppose for convenience that $k=r$, and we try to take the homogeneous forms $F_1,F_2,\ldots,F_r$ one at a time. For each $1\leq i\leq r-1$, it suffices to consider each case where
\begin{enumerate}
\item $V(F_1,\ldots,F_i)\subset \mb{P}^r$ has the expected dimension $r-i$
\item $V(F_1,\ldots,F_{i+1})$ is also dimension $r-i$,
\end{enumerate}
which means $V(F_{i+1})$ contains one of the components of $V(F_1,\ldots,F_i)$. By the definition of the Hilbert function, for a component $Z$ of $V(F_1,\ldots,F_i)$ , $V(F_{i+1})$ will contain $Z$ in codimension $h_Z(d_{i+1})$, where $h_Z$ is the Hilbert function of $Z$. 

At this point, the obstacle is that $Z$ could have a small Hilbert function, and it won't be enough to simply bound $Z$ below by the Hilbert function of a linear subspace. The next idea is to further divide up our case work in terms of the dimension of the \emph{linear span} of $Z$. 

For instance, if $Z$ spans all of $\mb{P}^r$, then we can bound $h_Z$ using Lemma \ref{ARBH}. If $Z$ spans a proper subspace $\Lambda\subset \mb{P}^r$, then we have a worse bound for the Hilbert function of $Z$. However, if we restrict $F_1,\ldots,F_{i+1}$ to $\Lambda\cong\mb{P}^b$ for $b<r$, we find $F_1|_{\Lambda},\ldots,F_{i+1}|_{\Lambda}$ are homogenous forms on $\mb{P}^b$ with intersection dimension 
\begin{align*}
\dim(V(F_1,\ldots,V_{i+1}))=r-i=(b-i)+(r-b),
\end{align*}
which is now $(r-b)+1$ more than expected. This means that, if we restrict our attention to $\Lambda$ and take the forms $F_1,\ldots,F_{i+1}$ one at a time, then there will be $(r-b)+1$ separate instances where intersecting with the next form won't drop the dimension, which will contribute $(r-b)+1$ times to the codimension instead of just once. 

Finally, in the final argument, it is cleaner to divide up the tuples of forms $F_1,\ldots,F_r$ for which $\dim(V(F_1,\ldots,F_r))\geq 1$ by the dimension $b$ of the spans of the curves they contain \emph{first}, and then in each case intersect the forms restricted to $b$-dimensional linear spaces $\Lambda \subset \mb{P}^r$ one at a time. 

\subsection{Funding}

This material is based upon work supported by the National Science Foundation Graduate
Research Fellowship Program under Grant No. 1144152. Any opinions, findings, and conclusions or recommendations expressed in this material are those of the author and do not necessarily reflect the views of the National Science Foundation.

\subsection{Acknowledgements}

The author would like to thank Joe Harris, Anand Patel, Bjorn Poonen, and Kaloyan Slavov for helpful conversations. The author would also like to thank the referees for helpful comments and patience with the author's inexperience. 

\section{Example}
\label{KEX}
The reader is encouraged to read the following example before the main argument because it will convey all of the main ideas. Throughout the example, we will be informal in our definitions and our justifications in the interest of clarity. We work over an algebraically closed field of arbitrary characteristic. 

We consider the following case of our general problem.
\begin{Pro}
What is the dimension of the space 4-tuples of homogenous forms in 5 variables and degree sequence (3,4,5,6), such that the common vanishing locus in $\mb{P}^4$ is positive dimensional, and what is a description of the component(s) of maximal dimension?
\end{Pro}

More precisely, let $W_d$ be the vector space of homogenous degree $d$ polynomials in $X_0,\ldots,X_4$. Define the closed subset $\Phi\subset W_3\times W_4\times W_5\times W_6$ to be the tuples $(F_3,F_4,F_5,F_6)$ where $F_i$ is of homogenous of degree $i$ and $F_3,F_4,F_5,F_6$ all vanish on some curve. Note that we are allowing the possibility that $F_i$ can be chosen to be zero. 

Since there are many types of curves that could be contained in $V(F_3,F_4,F_5,F_6)$, we set up the incidence correspondence
\begin{center}
\begin{tikzcd}
& \wt{\Phi} \arrow[swap, ld, "\pi_1"] \arrow[dr, "\pi_2"] &\\
\Phi & & \RHilb
\end{tikzcd}
\end{center}
where $\RHilb$ is open subset of the Hilbert scheme of curves parameterizing integral curves, and $\wt{\Phi}\subset \Phi\times \RHilb$ is the locus of pairs $((F_3,F_4,F_5,F_6),[C])$ such that $C$ is a curve contained in $V(F_3,F_4,F_5,F_6)$. 

\subsection{Filtration by span}
We write $\Phi=\Phi(1)\cup \Phi(2)\cup \Phi(3) \cup \Phi(4)$, where $\Phi(i)$ is the locus where $V(F_3,F_4,F_5,F_6)$ contains some integral curve $C$ spanning an $i$-dimensional plane. Even though the sets $\Phi(i)$ are in general only constructible sets, it still makes sense to talk about their dimensions, for example by taking the dimension of their closures. 

\subsubsection{Case of lines}
The set $\Phi(1)$ is by definition $\pi_1(\pi_2^{-1}(\{\text{lines}\}))$, or the locus where $V(F_3,F_4,F_5,F_6)$ contains some line. One can check the codimension of $\Phi(1)$ in $ W_3\times W_4\times W_5\times W_6$ is 
\begin{align*}
4+5+6+7-\dim(\mb{G}(1,4))=16.
\end{align*}

\subsubsection{Case of nondegenerate curves}
Next, we look at $\Phi(4)$. Since we will need to keep track of more data in the following discussion, we let $\Phi^{\text{curve}}_{3,4,5,6}(4):=\Phi(4)$. Here, the subscripts are the degrees of the homogenous forms $F_3,F_4,F_5,F_6$, and the superscript shows that we are restricting to the case where $V(F_3,F_4,F_5,F_6)$ contains a curve that spans a 4-D plane. 

If $(F_3,F_4,F_5,F_6)\in \Phi^{\text{curve}}_{3,4,5,6}(4)$, then in particular $F_3,F_4,F_5$ contain a nondegenerate integral curve. Under our new notation, this means forgetting the last homogenous form defines a map $\pi: \Phi^{\text{curve}}_{3,4,5,6}(4)\rightarrow \Phi^{\text{curve}}_{3,4,5}(4)$. Let $\Phi^{\text{surface}}_{3,4,5}(4)\subset \Phi^{\text{curve}}_{3,4,5}(4)$ be the loci of $(F_3,F_4,F_5)$ that all vanish on a nondegenerate integral surface. 

Since
\begin{align*}
\dim(\Phi^{\text{curve}}_{3,4,5,6}(4))=\max\{\dim(\pi^{-1}(\Phi^{\text{surface}}_{3,4,5}(4))),\dim(\pi^{-1}( \Phi^{\text{curve}}_{3,4,5}(4)\backslash \Phi^{\text{surface}}_{3,4,5}(4)))\},
\end{align*}
it suffices to bound $\pi^{-1}(\Phi^{\text{surface}}_{3,4,5}(4))$ and $\pi^{-1}( \Phi^{\text{curve}}_{3,4,5}(4)\backslash \Phi^{\text{surface}}_{3,4,5}(4))$ separately. Trivially, 
\begin{align*}
\dim(\pi^{-1}(\Phi^{\text{surface}}_{3,4,5}(4)))\leq \dim(\Phi^{\text{surface}}_{3,4,5}(4))+\dim(W_6).
\end{align*}
To bound $\pi^{-1}( \Phi^{\text{curve}}_{3,4,5}(4)\backslash \Phi^{\text{surface}}_{3,4,5}(4))$, we need to bound the dimension of the fiber of 
\begin{align*}
\pi: \pi^{-1}( \Phi^{\text{curve}}_{3,4,5}(4)\backslash \Phi^{\text{surface}}_{3,4,5}(4))\rightarrow \Phi^{\text{curve}}_{3,4,5}(4)\backslash \Phi^{\text{surface}}_{3,4,5}(4)
\end{align*}
 from below. Suppose $(F_3,F_4,F_5,F_6)\in \pi^{-1}( \Phi^{\text{curve}}_{3,4,5}(4)\backslash \Phi^{\text{surface}}_{3,4,5}(4))$. Then, $V(F_3, F_4, F_5)$ contains no nondegenerate surfaces, but it does contain a nondegenerate curve $C$. We see that the number of conditions for $F_6$ to contain $C$ set theoretically is at least $6\cdot 4+1=25$, which is the Hilbert function of the rational normal curve in $\mb{P}^4$ evaluated at 6. 

To summarize what we have so far, it is easier to use codimensions instead of dimensions. Let $\codim(\Phi^{\text{curve}}_{3,4,5,6}(4))$ denote the codimension in $W_3\times W_4\times W_5 \times W_6$. Similarly, $\codim(\Phi^{\text{curve}}_{3,4,5}(4))$ and $\codim(\Phi^{\text{surface}}_{3,4,5}(4))$ denote the codimension in $W_3\times W_4\times W_5$. So far, we have shown
\begin{align*}
\codim(\Phi^{\text{curve}}_{3,4,5,6}(4))&\geq \min\{\codim(\Phi^{\text{surface}}_{3,4,5}(4)),\codim(\Phi^{\text{curve}}_{3,4,5}(4))+25\}\\
&=\min\{\codim(\Phi^{\text{surface}}_{3,4,5}(4)),25\}.
\end{align*}
Repeating the same argument above for $\Phi^{\text{surface}}_{3,4,5}(4)$, we see
\begin{align*}
\codim(\Phi^{\text{surface}}_{3,4,5}(4))\geq \min\{\codim(\Phi^{\text{3-fold}}_{3,4}(4)),51\},
\end{align*}
where $51=2\binom{2+5-1}{5-1}+\binom{2+5}{5}$ is the Hilbert function of a minimal surface evaluated at 5. Continuing the process, we find
\begin{align*}
\codim(\Phi^{\text{3-fold}}_{3,4}(4))&\geq \min\{\codim(\Phi^{\text{4-fold}}_{3}(4)),55\}.
\end{align*}
As a possible point of confusion, the number of conditions it is for a degree 4 hypersurface to contain a degree 3 hypersurface is $\binom{4+4}{4}-\binom{4+1}{1}=65$, which is greater than 55. However, to keep our arguments consistent, we instead choose the weaker bound by the Hilbert function of a quadric hypersurface, which is also a variety of minimal degree. Finally,
\begin{align*}
\codim(\Phi^{\text{4-fold}}_{3}(4))&=35,
\end{align*}
where $\Phi^{\text{4-fold}}_{3}(4)$ is just the single point in $W_3$ corresponding to the zero homogenous form. Putting everything together, we find
\begin{align*}
\codim(\Phi(4))=\codim(\Phi^{\text{curve}}_{3,4,5,6}(4))\geq \min\{25,51,55,35\}=25. 
\end{align*}
\subsubsection{Case of curves spanning a 3-plane}
\label{ICEX}
In order to bound the codimension of $\Phi(3)$ in $W_3\times W_4\times W_5\times W_6$, we  use the same argument as we used to bound the codimension of $\Phi(4)$ with one small extra step. Consider the incidence correspondence
\begin{center}
\begin{tikzcd}
& \wt{\Phi}(3) \arrow[swap, ld, "\pi_1"] \arrow[dr, "\pi_2"] &\\
\Phi(3) & & \mb{G}(3,4)=(\mb{P}^4)^{*}
\end{tikzcd}
\end{center}
where $\wt{\Phi}(3)\subset \Phi(3)\times (\mb{P}^4)^{*}$ consists of pairs $((F_3,F_4,F_5,F_6),H)$ where $H$ is a hyperplane such that $V(F_3, F_4, F_5, F_6)$ contains an integral curve that spans $H$. To bound the codimension of $\Phi(3)$ from below, it suffices to bound the codimension of the fibers of $\pi_2$ in $W_3\times W_4\times W_5\times W_6$ from below. 

By restricting $F_3,F_4,F_5,F_6$ to $H$, we see that the codimension of each fiber of $\pi_2$ is precisely the number of conditions for homogenous forms of degrees (3,4,5,6) in $\mb{P}^3$ to vanish on some nondegenerate curve in $\mb{P}^3$. Writing this in symbols, the codimension of the fibers of $\pi_2$ is the codimension of $\Phi^{3,\text{curve}}_{3,4,5,6}(3)\subset W_{3,3}\times W_{3,4}\times W_{3,5}\times W_{3,6}$, where $W_{r,d}\cong \mb{A}^{\binom{r+d}{d}}$ is the vector space of homogenous polynomials of degree $d$ in $r+1$ variables, and $\Phi^{\mb{P}^3,\text{curve}}_{3,4,5,6}(3)$ is the locus of tuples of homogenous forms $(F_3',F_4',F_5',F_6')$ vanishing on some nondegenerate integral curve in $\mb{P}^3$. The extra superscript is to remind ourselves that we are working in $\mb{P}^3$ now. 

Repeating the same argument as we did to bound $\codim(\Phi(4))$ above, we compute
\begin{align*}
\codim(\Phi^{\mb{P}^3,\text{curve}}_{3,4,5,6}(3))&\geq \min\{\codim(\Phi^{\mb{P}^3,\text{surface}}_{3,4,5}(3)),19+\codim(\Phi^{\mb{P}^3,\text{curve}}_{3,4,5}(3))\}\\
\codim(\Phi^{\mb{P}^3,\text{curve}}_{3,4,5}(3))&\geq \min\{\codim(\Phi^{\mb{P}^3,\text{surface}}_{3,4}(3)),16\}\\
\codim(\Phi^{\mb{P}^3,\text{surface}}_{3,4,5}(3))&\geq \min\{\codim(\Phi^{\mb{P}^3,\text{3-fold}}_{3,4}(3)),\codim(\Phi^{\mb{P}^3,\text{surface}}_{3,4}(3))+36\}\\
\codim(\Phi^{\mb{P}^3,\text{surface}}_{3,4}(3))&\geq \min\{\codim(\Phi^{\mb{P}^3,\text{3-fold}}_{3}(3)),25\}\\
\codim(\Phi^{\mb{P}^3,\text{3-fold}}_{3,4}(3))&=\codim(\Phi^{\mb{P}^3,\text{3-fold}}_{3}(3))+35\\
\codim(\Phi^{\mb{P}^3,\text{3-fold}}_{3}(3))&=20.
\end{align*}
Putting everything together, we see
\begin{align*}
\codim(\Phi^{\mb{P}^3,\text{curve}}_{3,4,5,6}(3))&\geq \min\{19+16,19+25,19+20,36+25,36+20,35+20\}=35
\end{align*}
Finally, we conclude 
\begin{align*}
\codim(\Phi(3))&\geq -\dim(\mb{G}(3,4))+\codim(\Phi^{\mb{P}^3,\text{curve}}_{3,4,5,6}(3))\geq 31. 
\end{align*}
\subsubsection{Case of plane curves}
The case of plane curves can be done directly since we can completely classify them. However, if we apply the same argument as above, the bound we get is $\codim(\Phi(2))$ is at least
\begin{align*}
-\dim(\mb{G}(2,4))+\min\{13+11+9,13+11+10,13+15+10,21+15+10\}=27.
\end{align*}
\subsubsection{Combining the cases}
Combining the cases, we get
\begin{align*}
\codim(\Phi^1)&=16,\ \codim(\Phi^2)\geq 27,\ \codim(\Phi^3)\geq 31,\ \codim(\Phi^4)\geq 25.
\end{align*}
Summarizing, we now know that the codimension of $\Phi$ in $W_3\times W_4\times W_5\times W_6$ is exactly 16, the component of maximal dimension of $\Phi$ corresponds to tuples of homogenous forms vanishing on some line, and a component of second largest dimension has codimension at least 25. 

\section{General Argument}
\label{MainArgument}

We now implement the argument given in Section \ref{KEX} in the general case. The notation is heavy, so we have included a list of conventions and a chart of the definitions introduced for the argument in this section for reference.

Conventions:
\begin{enumerate}
\item the base field is an algebraically closed field $K$ of arbitrary characteristic
\item the ambient projective space is $\mb{P}^r$
\item $X$ is a subvariety of $\mb{P}^r$
\item $A$ is a constructible set
\item $\mc{X}\to S$ is a family over a finite type $K$-scheme $S$
\item $F_i$ is a homogeous form of degree $d_i$
\item $D$ is an integer greater than $d_1\cdots d_k$
\end{enumerate}

\pagebreak
Notations:
\begin{center}
\begin{longtable}{l | p{11cm} | r}
Symbol & Informal Meaning & Definition \\\hline
$\Hilb_X$ & Hilbert scheme of $X$ & \ref{HilbD}\\
$\Hilb^b_X$ & $\Hilb_X$ restricted to subschemes of dimension $b$ & \ref{HilbD}\\
$\RHilb_X$ & $\Hilb_X$ of $X$ restricted to integral subschemes & \ref{RHilbD}\\
$\RHilb^b_X$ & $\RHilb_X\cap \Hilb^b_X$  & \ref{RHilbD}\\
$\RHilb_X^{b,\leq D}$ & $\RHilb^b_X$ restricted to schemes of degree at most $b$ & \ref{HilbRDD}\\
$\Hilb_{\mc{X}/S}$ & relative version of $\Hilb_X$ & \ref{HilbS}\\
$\Hilb^b_{\mc{X}/S}$ & relative version of $\Hilb^b_X$ & \ref{HilbS}\\
$\RHilb_{\mc{X}/S}$ & relative version of $\RHilb_X$  & \ref{HilbS}\\
$\RHilb^b_{\mc{X}/S}$ & relative version of $\RHilb^b_X$  & \ref{HilbS}\\
$\RHilb_{\mc{X}/S}^{b,\leq D}$ & relative version of $\Hilb_X^{b,\leq D}$ & \ref{HilbRDD}\\
$W_{r,d}$ & vector space degree $d$ homogeous forms in $r+1$ variables & \ref{WDEF}\\
$F_{r,d}$ & universal hypersurface $F_{r,d}\subset W_{r,d}\times \mb{P}^r$ & \ref{FDEF}\\
$\Phi^{\mb{P}^r,a}_{d_1,\ldots,d_k}(X)$ & parameterizes forms $F_1,\ldots,F_k$ whose vanishing locus in $X\subset \mb{P}^r$ is dimension $a$ more than expected & \ref{PX}\\
$\Phi^{\mb{P}_S^r,a}_{d_1,\ldots,d_k}(\mc{X}/S)$ & relative version of $\Phi^{\mb{P}^r,a}_{d_1,\ldots,d_k}(X)$  & \ref{PXA}\\
$\wt{\Phi}^{\mb{P}^r,a}_{d_1,\ldots,d_k}(X)$ & parameterizes forms $F_1,\ldots,F_k$ in $\Phi^{\mb{P}^r,a}_{d_1,\ldots,d_k}(X)$ together with a choice of integral subscheme $Z\subset X\cap V(F_1,\ldots,F_k)$ of dimension $a$ more than expected  & \ref{PXT}\\
$\wt{\Phi}^{\mb{P}^r_S,a}_{d_1,\ldots,d_k}(\mc{X}/S)$ & relative version of $\wt{\Phi}^{\mb{P}^r,a}_{d_1,\ldots,d_k}(X)$ & \ref{PXAT}\\
$\wt{\Phi}^{\mb{P}^r,a}_{d_1,\ldots,d_k}(X,A)$ & $\wt{\Phi}^{\mb{P}^r,a}_{d_1,\ldots,d_k}(X)$ restricted to choices $((F_1,\ldots,F_k),Z)$ where $Z\in A$ ($A\subset \RHilb_{\mb{P}^r}$) & \ref{IMA}\\
$\Phi^{\mb{P}^r,a}_{d_1,\ldots,d_k}(X,A)$ & image $\wt{\Phi}^{\mb{P}^r,a}_{d_1,\ldots,d_k}(X)\to \Phi^{\mb{P}^r,a}_{d_1,\ldots,d_k}(X,A)$ forgetting $Z$ & \ref{IMA}\\
$\wt{\Phi}^{\mb{P}^r,a}_{d_1,\ldots,d_k}(\mc{X}/S,A)$ & relative version of $\wt{\Phi}^{\mb{P}^r,a}_{d_1,\ldots,d_k}(X,A)$  & \ref{IMAS}\\
$\Phi^{\mb{P}^r,a}_{d_1,\ldots,d_k}(\mc{X}/S,A)$ & relative version of $\Phi^{\mb{P}^r,a}_{d_1,\ldots,d_k}(X,A)$ & \ref{IMAS}\\
$h_A$ & $h_A(d)$ is the minimum of Hilbert functions $h_Z(d)$ for $[Z]\in A\subset \wt{\Hilb_X}$ \ref{minhD}\\
$\Co(A)$ & all $[Y]\in \RHilb_X$ where $Y\supset Z$ for $[Z]\in A$ & \ref{containD}\\
$\Span(r,b)$ & subset of $\RHilb_{\mb{P}^r}$ of schemes $Z$ that span a $b$-dimensional linear space & \ref{spanD}\\
\end{longtable}
\end{center}

\subsection{Definitions}
Here, we fix the notation for the objects of study. 

\subsubsection{Hilbert scheme and relative Hilbert scheme}
We won't use the geometry of the Hilbert scheme, but we will use its existence. 
\begin{Def}
\label{HilbD}
If $X\subset \mb{P}^r$ is a projective scheme, let $\Hilb_X$ denote the Hilbert scheme of subschemes of $X$. Let $\Hilb_X^b\subset \Hilb_X$ denote the connected components corresponding to subschemes of dimension $b$.
\end{Def}

\begin{Def}
\label{RHilbD}
Let $\RHilb_X\subset \Hilb_X$ denote the subset of the Hilbert scheme corresponding to geometrically integral subschemes. Recall $\RHilb$ is open in $\Hilb$ \cite[IV 12.1.1 (x)]{EGA}, . Let $\RHilb_X^b\subset \RHilb_X$ correspond to subschemes of dimension $b$. 
\end{Def}

\begin{Def}
\label{HilbS}
Suppose $S$ is a finite type $K$-scheme and $\mc{X}\subset \mb{P}^r_S$ is a closed subscheme. Let $\Hilb_{\mc{X}/S}$ denote the relative Hilbert scheme of the family $\mc{X}\rightarrow S$. Similar to above, we let $\RHilb_{\mc{X}/S}\subset \Hilb_{\mc{X}/S}$ denote the open locus parameterizing geometrically integral subschemes, and $\RHilb_{\mc{X}/S}^b\subset \Hilb_{\mc{X}/S}^b$ denote the restriction to $b$-dimensional subschemes. 
\end{Def}

\subsubsection{Constructible sets}
\label{CONSET}
We will need to work with dimensions and maps of constructible sets. Chevalley's theorem implies constructible subsets remain constructible after taking images \cite[Tag 054K]{stacks-project}. Since we are only interested in their dimensions, it suffices to think of them as subsets of an ambient space with a notion of dimension. 
\begin{Def}
If $X$ is a scheme, then a constructible set $A\subset X$ is a finite union of locally closed subsets. 
\end{Def}

To take the dimension of a constructible set, it suffices to either look at the generic points or take the closure.
\begin{Def}
\label{DEFDIM}
If $A\subset X$ is a constructible set, then $\dim(A):=\dim(\overline{A})$. The dimension of the empty set is $-\infty$. The codimension of the empty set in a nonempty constructible set is $\infty$. 
\end{Def}

By applying the usual theorem on fiber dimension to $\overline{A}$ and $\overline{B}$ below, we have:
\begin{Lem}
\label{FY6}
If $f:X\rightarrow Y$ is a morphism of finite type schemes over a field, $A\subset X$ and $B=f(A)\subset Y$ constructible sets, and $\dim(f^{-1}(b)\cap A)<c$ for all $b\in B$, then 
\begin{align*}
\dim(A)\leq \dim(B)+c. 
\end{align*}
If $\dim(f^{-1}(b))=c$ for all $b\in B$, then equality holds. 
\end{Lem}

For notational convenience, we define the pullback of a constructible set.
\begin{Def}
If $X\rightarrow Z \leftarrow Y$ are morphisms and $A\subset X$ is a constructible set, then define $A\times_Z Y$ to be the preimage of $A$ in $X\times_Z Y$ under $X\times_Z Y\rightarrow X$. 
\end{Def}
\subsubsection{Locus of tuples of hypersurfaces not intersecting properly}
We parameterize our space of homogenous forms with affine spaces. 

\begin{Def}
\label{WDEF}
Given positive integers $r,d$, let $W_{r,d}\cong \mb{A}^{\binom{r+d}{d}}$ be the affine space whose underlying vector space is $H^0(\mb{P}^r,\mathscr{O}_{\mb{P}^r}(d))$, the hypersurfaces of degree $d$ in $\mb{P}^r$. 
\end{Def}

\begin{Def}
\label{FDEF}
Over $W_{r,d}$, let $\mc{F}_{r,d}\subset \mb{P}^r\times W_{r,d}$ be the universal family $\mc{F}_{r,d}\rightarrow W_{r,d}$. 
\end{Def}

\begin{Def}
\label{PX}
If $X\subset \mb{P}^r$ is a projective scheme, $a$ a nonnegative integer, and $(d_1,\ldots,d_k)$ is a tuple of positive integers, define $\Phi^{\mb{P}^r,a}_{d_1,\ldots,d_k}(X)\subset \prod_{i=1}^{k}W_{r,d_i}$ to be the closed subset of tuples $(F_1,\ldots,F_k)$ of homogenous forms of degrees $(d_1\,\ldots,d_k)$ such that 
\begin{align*}
\dim(X\cap V(F_1,\ldots,F_k))\geq \dim(X)-k+a.
\end{align*}
\end{Def}

In particular, $\Phi^{\mb{P}^r,0}_{d_1,\ldots,d_k}(X)$ is all of $\prod_{i=1}^{k}W_{r,d_i}$ and $\Phi^{\mb{P}^r,1}_{d_1,\ldots,d_k}(\mb{P}^r)$ is the locus of hypersurfaces failing to be a complete intersection. Similarly, given a family $\mc{X}\subset \mb{P}^r_S$, we can define a relative version.

In Definition \ref{PX} it is useful to keep in mind that, even though we are mostly interested in the case of hypersurfaces failing to be a complete intersection $\Phi^{\mb{P}^r,1}_{d_1,\ldots,d_k}(\mb{P}^r)$, we will need to let $X$ and $a$ vary.

\begin{Def}
\label{PXA}
Given a finite type $K$-scheme $S$ and a closed subscheme $\mc{X}\subset \mb{P}^r_S$, let $\Phi^{\mb{P}_S^r,a}_{d_1,\ldots,d_k}(\mc{X}/S)\subset S\times_K \prod_{i=1}^{k}{W_{r,d_i}}$ denote the closed subset such that for all $s\in S$ with residue field $k(s)$, then the fiber of $\Phi^{\mb{P}_S^r,a}_{d_1,\ldots,d_k}(\mc{X}/S)\rightarrow S$ over $s$ is $\Phi^{\mb{P}_{k(s)}^r,a}_{d_1,\ldots,d_k}(\mc{X}|_s)$. 
\end{Def}

Both $\Phi^{\mb{P}^r,a}_{d_1,\ldots,d_k}(X)$ and $\Phi^{\mb{P}_S^r,a}_{d_1,\ldots,d_k}(\mc{X}/S)$ can be constructed by applying upper semicontinuity of dimension to the respective families over $\prod_{i=1}^{k}{W_{r,d_i}}$ and $S\times \prod_{i=1}^{k}{W_{r,d_i}}$.

\subsection{Incidence correspondence}
We will want to break up $\Phi^{\mb{P}^r,1}_{d_1,\ldots,d_k}(\mb{P}^r)$ into constructible sets based on the types of schemes contained in the common vanishing loci of the $k$-tuples of homogenous forms. To formalize this, we define the following.
\begin{Def}
\label{PXT}
Given $X\subset \mb{P}^r$ projective, let $\wt{\Phi}^{\mb{P}^r,a}_{d_1,\ldots,d_k}(X)\subset \Phi^{\mb{P}^r,a}_{d_1,\ldots,d_k}(X)\times \RHilb_X^{\dim(X)-k+a}$ denote the closed subset corresponding to pairs $((F_1,\ldots,F_k),[V])$, where $(F_1,\ldots,F_k)$ is in $\prod_{i=1}^{k}{W_{r,d_i}}$ and $[V]$ is in $ \RHilb_{\dim(X)-k+a}$ such that $V\subset X\cap V(F_1,\cdots F_k)$. 

\end{Def}
Similarly, we have the relative version.
\begin{Def}
\label{PXAT}
Given a finite type $K$-scheme $S$ and a subscheme $\mc{X}\subset \mb{P}^r_S$ where each fiber of $\mc{X}/S$ has the same dimension $b$, let $\wt{\Phi}^{\mb{P}^r_S,a}_{d_1,\ldots,d_k}(\mc{X}/S)\subset \Phi^{\mb{P}_S^r,a}_{d_1,\ldots,d_k}(\mc{X}/S)\times \RHilb_{\mc{X}/S}^{b-k+a}$ denote the closed subset such that, for each $s\in S$ with residue field $k(s)$, the restriction of $\wt{\Phi}^{\mb{P}^r_S,a}_{d_1,\ldots,d_k}(\mc{X}/S)\rightarrow S$ to the fiber over $s$ is $\wt{\Phi}^{\mb{P}_{k(s)}^r,a}_{d_1,\ldots,d_k}(\mc{X}|_{s})$.
\end{Def}
To construct both $\wt{\Phi}^{\mb{P}^r,a}_{d_1,\ldots,d_k}(X)$ and $\wt{\Phi}^{\mb{P}^r_S,a}_{d_1,\ldots,d_k}(\mc{X}/S)$, we can use:
\begin{Lem}
[{\cite[Lemma 7.1]{ACGH2}}]
\label{CFF}
Given a scheme $A$ and two closed subschemes $B,C\subset \mb{P}^r_A$ with $B$ flat over $A$, there exists a closed subscheme $D\subset A$ such that any morphism $T\rightarrow A$ factors through $D$ if and only if $B\times_A T$ is a subscheme of $C\times_A T$.
\end{Lem} 
In particular, $\wt{\Phi}^{\mb{P}^r,a}_{d_1,\ldots,d_k}(X)$ and $\wt{\Phi}^{\mb{P}^r_S,a}_{d_1,\ldots,d_k}(\mc{X}/S)$ actually have a canonical scheme theoretic structure, though we will not make use of it. 

To find the maximal dimensional components of $\Phi^{\mb{P}^r,1}_{d_1,\ldots,d_k}(\mb{P}^r)$, our general plan is to cover $\RHilb_{\mb{P}^r}^{r-k+1}$ with constructible sets $A_i$, and bound the dimensions of $\pi_1(\pi_2^{-1}(A_i))$ for each $i$, where $\pi_1$ and $\pi_2$ are given by the following diagram in the case $X=\mb{P}^r$ and $a=1$. 
\begin{center}
\begin{tikzcd}
& \wt{\Phi}^{\mb{P}^r,a}_{d_1,\ldots,d_k}(X) \arrow[swap, ld, "\pi_1"] \arrow[dr, "\pi_2"] &\\
\Phi^{\mb{P}^r,a}_{d_1,\ldots,d_k}(X) & & \RHilb_{X}^{\dim(X)-k+a}
\end{tikzcd}
\end{center}
This isn't a serious issue, but we would like to restrict ourselves to finitely many connected components of $\RHilb_{X}^{\dim(X)-k+a}$. Otherwise, for example, if we take the subset $A\subset \RHilb_{X}^{\dim(X)-k+a}$ of nondegenerate varieties, then $\pi_1(\pi_2^{-1}(A))$ is a priori only a countable union of constructible sets obtained by applying Chevalley's theorem to $A$ restricted to each connected component of $\RHilb_{X}^{\dim(X)-k+a}$. 

Therefore, we define
\begin{Def}
\label{HilbRDD}
Let $\RHilb_X^{b,\leq D}\subset \RHilb_X^b$ denote the connected components parameterizing subschemes of degree at most $D$. We similarly define $\RHilb_{\mc{X}/S}^{b,\leq D}\subset \RHilb_{\mc{X}/S}^{b}$ for $\mc{X}\subset \mb{P}^r_S$. 
\end{Def}

From Chow's finiteness theorem \cite[Exercise I.3.28 and Theorem I.6.3]{Kollar}, we see $\RHilb_X^{c,\leq D}$ and $\RHilb_{\mc{X}/S}^{c,\leq D}$ have only finitely many connected components. From refined Bezout's theorem \cite[Example 12.3.1]{Fulton}, for $D=d_1\cdots d_k$, $\pi_1(\pi_2^{-1}(\RHilb_X^{\dim(X)-k+a,\leq D}))$ is all of $\Phi^{\mb{P}^r,a}_{d_1,\ldots,d_k}(X)$. In general, we will always choose $D$ such that $D\geq d_1\cdots d_k$. 

\begin{Rem}
The application of Chow's finiteness theorem and refined Bezout's theorem are unnecessary and entirely for notational convenience. Very generally, given any finite dimensional Noetherian scheme $S$ covered by a family of constructible sets, there will exist a finite subcover by looking at the generic points of $S$ and applying induction on dimension. Applying this to $\Phi^{\mb{P}^r,a}_{d_1,\ldots,d_k}(X)$ shows that there is a union $A$ of finitely many connected components of $\RHilb_{\mb{P}^r}^{r-k+1}$ such that $\pi_1(\pi_2^{-1}(A))=\Phi^{\mb{P}^r,a}_{d_1,\ldots,d_k}(X)$, so we can just restrict our incidence correspondence to $A$ instead of $\RHilb_X^{\dim(X)-k+a,\leq D}$. 
\end{Rem}

Motivated from above, we make the following definition.
\begin{Def}
\label{IMA}
For $A\subset \RHilb_{\mb{P}^r}$ a constructible subset and $D\geq d_1\cdots d_k$ an integer, let $\wt{\Phi}^{\mb{P}^r,a}_{d_1,\ldots,d_k}(X,A)=\pi_2^{-1}(A\cap \RHilb_X^{\dim(X)-k+a,\leq D})$ and $\Phi^{\mb{P}^r,a}_{d_1,\ldots,d_k}(X,A)=\pi_1(\pi_2^{-1}(A\cap \RHilb_X^{\dim(X)-k+a,\leq D}))$.
\end{Def}

Similarly, we have the relative version.

\begin{Def}
\label{IMAS}
Let  $S$ be a finite type $K$-scheme, $\mc{X}\subset \mb{P}^r_S$ a family such that $\mc{X}\rightarrow S$ has $b$-dimensional fibers, $\mc{A}\subset S\times \RHilb_{\mb{P}^r}$ a constructible subset. For $D\geq d_1\cdots d_k$, let $\wt{\Phi}^{\mb{P}_S^r,a}_{d_1,\ldots,d_k}(\mc{X}/S,\mc{A})=\pi_2^{-1}(\mc{A}\cap \RHilb_{\mc{X}/S}^{b-k+a,\leq D})$ and $\Phi^{\mb{P}_S^r,a}_{d_1,\ldots,d_k}(\mc{X}/S,\mc{A})=\pi_1(\pi_2^{-1}(\mc{A}\cap \RHilb_{\mc{X}/S}^{b-k+a,\leq D}))$ for $\pi_1$ and $\pi_2$ defined as below.
\begin{center}
\begin{tikzcd}
& \wt{\Phi}^{\mb{P}_S^r,a}_{d_1,\ldots,d_k}(\mc{X}/S) \arrow[swap, ld, "\pi_1"] \arrow[dr, "\pi_2"] &\\
\Phi^{\mb{P}_S^r,a}_{d_1,\ldots,d_k}(\mc{X}/S) & & \RHilb_{\mc{X}/S}^{b-k+a,\leq D}
\end{tikzcd}
\end{center}
\end{Def}

\subsection{Inducting on the number of hypersurfaces}
We present Lemma \ref{kind} which bounds the dimension of $\Phi^{\mb{P}^r,a}_{d_1,\ldots,d_k}(X,A)$.
\subsubsection{Preliminary definitions}
Instead of working with dimensions, it is easier to work with codimensions.
\begin{Def}
Let $\codim(\Phi^{\mb{P}^r,a}_{d_1,\ldots,d_k}(X))$ and $\codim(\Phi^{\mb{P}^r,a}_{d_1,\ldots,d_k}(X,A))$ be the codimension in $\prod_{i=1}^{k}W_{r,d_i}$. 
We similarly let $\codim(\Phi^{\mb{P}_S^r,a}_{d_1,\ldots,d_k}(\mc{X}/S))$ and $\codim(\Phi^{\mb{P}_S^r,a}_{d_1,\ldots,d_k}(\mc{X}/S,\mc{A}))$ mean the codimension in $S\times\prod_{i=1}^{k}W_{r,d_i}$.
\end{Def}
We will need to refer to the minimum of the Hilbert function over a constructible set of the Hilbert scheme.
\begin{Def}
\label{minhD}
For $A\subset \RHilb_X$ a constructible set, let
\begin{align*}
h_A(d):=\min\{h_Z(d):[Z]\in A\},
\end{align*}
where $h_Z$ refers to the Hilbert function of a  projective scheme $Z$. This is well-defined since $Z\subset X\subset \mb{P}^r$.
\end{Def}
We will need notation to refer to varieties containing a piece of the Hilbert scheme.
\begin{Def}
\label{containD}
Given $A\subset \RHilb^{\leq D}_{X}$ a constructible subset, let $\Co(A)$ denote the constructible subset consisting of all $[Y]\in \RHilb^{\leq D}_X$ such that $Y$ contains some $Z$ for $[Z]\in A$. 
\end{Def}

To see $\Co(A)$ is constructible, let $\mc{Z}\subset \RHilb^{\leq D}_{X}\times \RHilb^{\leq D}_{X}$ denote the subset corresponding to pairs $([Y],[Z])$ such that $Z\subset Y$. Lemma \ref{CFF} shows $\mc{Z}$ is closed. From the incidence correspondence
\begin{center}
\begin{tikzcd}
& \mc{Z} \arrow[swap, ld, "\pi_1"] \arrow[dr, "\pi_2"] &\\
\RHilb^{\leq D}_{X}& & \RHilb^{\leq D}_{X}
\end{tikzcd}
\end{center}
we see $\Co(A)=\pi_1(\pi_2^{-1}(A))$ is constructible by Chevalley's theorem \cite[Tag 054K]{stacks-project}. 
\subsubsection{Forgetting a hypersurface}
\begin{Lem}
\label{kind}
For $X\subset \mb{P}^r$ a projective scheme and $A\subset \RHilb^{\dim(X)-k+a,\leq D}_X$ constructible,
\begin{align*}
\codim(\Phi^{\mb{P}^r,a}_{d_1,\ldots,d_k}(X,A))\geq \min\{\codim(\Phi^{\mb{P}^r,a}_{d_1,\ldots,d_{k-1}}(X,A_1)),\codim(\Phi^{\mb{P}^r,a-1}_{d_1,\ldots,d_{k-1}}(X,A))+h_A(d_k)\},
\end{align*}
where $A_1=\Co(A)\cap \RHilb_X^{\dim(X)-k+a+1,\leq D}$.
\end{Lem}

In words, the idea of Lemma \ref{kind} is intuitively obvious. If $(F_1,\ldots,F_k)$ are homogenous forms that vanish on $Z$ for some $[Z]\in A$, then at least one of the following holds:
\begin{enumerate}
\item
$(F_1,\ldots,F_{k-1})$ vanish on some $Z'\supset Z$ with $\dim(Z')=\dim(Z)+1$
\item
$\dim(V(F_1,\ldots,F_{k-1}))=\dim(V(F_1,\ldots,F_k))$. 
\end{enumerate}
The first case is captured by the codimension of $\Phi^{\mb{P}^r,a}_{d_1,\ldots,d_{k-1}}(X,A_1)$, and the second case is bounded above by $\codim(\Phi^{\mb{P}^r,a-1}_{d_1,\ldots,d_{k-1}}(X,A))+h_A(d_k)$, since $F_k$ must vanish on one of the components of $V(F_1,\ldots,F_{k-1})$ which happens in codimension at least $h_A(d_k)$. 

\begin{proof}
Let the constructible set $B$ be the image of the projection $\pi: \Phi^{\mb{P}^r,a}_{d_1,\ldots,d_k}(X,A)\rightarrow \prod_{i=1}^{k-1}{W_{r,d_i}}$ forgetting the last hypersurface. By definition, we see $B\subset \Phi^{\mb{P}^r,a}_{d_1,\ldots,d_{k-1}}(X,A_1)\cup \Phi^{\mb{P}^r,a-1}_{d_1,\ldots,d_{k-1}}(X,A_2)$. 

Lemma \ref{FY6} shows
\begin{align*}
\codim(\pi^{-1}(\Phi^{\mb{P}^r,a}_{d_1,\ldots,d_{k-1}}(X,A_1)))\geq \codim(\Phi^{\mb{P}^r,a}_{d_1,\ldots,d_{k-1}}(X,A_1)). 
\end{align*}
To bound $\pi^{-1}(\Phi^{\mb{P}^r,a-1}_{d_1,\ldots,d_{k-1}}(X,A_2)\backslash \Phi^{\mb{P}^r,a}_{d_1,\ldots,d_{k-1}}(X,A_1))$, let 
\begin{align*}
(F_1,\ldots,F_{k-1})\in \Phi^{\mb{P}^r,a-1}_{d_1,\ldots,d_{k-1}}(X,A_2)\backslash \Phi^{\mb{P}^r,a}_{d_1,\ldots,d_{k-1}}(X,A_1)
\end{align*}
be a closed point and $V_1,\ldots,V_\ell$ be the $\dim(X)-k+a$ dimensional components of $X\cap V(F_1,\ldots, F_{k-1})$ that are in $A$. Here, we are using $(F_1,\ldots, F_{k-1})\notin \Phi^{\mb{P}^r,a}_{d_1,\ldots,d_{k-1}}(X,A_1)$ so that $X\cap V(F_1,\ldots F_{k-1})$ does not contain any component $V$ of dimension greater than $\dim(X)-k+a$ containing $Z$ for $[Z]\in A$. 

The locus of $F_k$ that contain $V_i$ is a linear subspace in $W_{r,d_k}$ of codimension $h_{V_i}(d_k)$. Therefore, the codimension of $\pi^{-1}(p)$ in $W_{r,d_k}$ is at least $h_A(d_k)$. Therefore, we can apply Lemma \ref{FY6} to conclude
\begin{align*}
\codim(\pi^{-1}(\Phi^{\mb{P}^r,a-1}_{d_1,\ldots,d_{k-1}}(X,A_2)\backslash \Phi^{\mb{P}^r,a}_{d_1,\ldots,d_{k-1}}(X,A_1)))&\geq\\
\codim(\Phi^{\mb{P}^r,a-1}_{d_1,\ldots,d_{k-1}}(X,A_2)\backslash \Phi^{\mb{P}^r,a}_{d_1,\ldots,d_{k-1}}(X,A_1))+h_A(d_k)&\geq\\
\codim(\Phi^{\mb{P}^r,a-1}_{d_1,\ldots,d_{k-1}}(X,A_2))+h_A(d_k)&.
\end{align*}
\end{proof}

\subsection{Varieties contained in a family}
Now, we want to partition by span as outlined in Sections \ref{PI} and \ref{KEX}. Since we will need to vary a linear space $\Lambda \subset \mb{P}^r$, we will need relative versions of $\Phi^{\mb{P}^r,a}_{d_1,\ldots,d_k}(X)$, where we will eventually set $X=\Lambda$ and let $\Lambda$ vary. Even though we only need it for a very special case, it is easier only the notation to introduce the definitions more generally.

\begin{Lem}
\label{IC}
Suppose $S$ is a finite type $K$-scheme, $X\subset \mb{P}^r_K$ is a projective scheme, $\mc{X}\subset X\times_K S$ is a closed subscheme such that each fiber of $\mc{X}\rightarrow S$ has codimension $e>0$ in $X$, $\mc{A}\subset \RHilb_{\mc{X}/S}^{\dim(X)-k+a,\leq D}$ is a constructible subset, the image of $\mc{A}$ in $\RHilb_X^{\dim(X)-k+a}$ under $\RHilb_{\mc{X}/S}^{\dim(X)-k+a}\rightarrow \RHilb_X^{\dim(X)-k+a}$ is $A$, and the fiber dimension $\mc{A}\rightarrow A$ is at least $c$ for all points in $A$. Then,
\begin{align}
\codim(\Phi^{\mb{P}^r,a}_{d_1,\ldots,d_k}(X,A))\geq \codim(\Phi^{\mb{P}_S^r,a+e}_{d_1,\ldots,d_k}(\mc{X}/S,\mc{A}))-\dim(S)+c, \label{ICeq}
\end{align}
and equality holds if the fiber dimension of $\mc{A}\rightarrow A$ is exactly $c$ for all points of $A$, $\wt{\Phi}^{\mb{P}^r,a}_{d_1,\ldots,d_k}(X,A)\rightarrow \Phi^{\mb{P}^r,a}_{d_1,\ldots,d_k}(X,A)$ has generically finite fibers, and $\Phi^{\mb{P}_S^r,a+e}_{d_1,\ldots,d_k}(\mc{X}/S,\mc{A})$ is irreducible. 
\end{Lem}

Like Lemma \ref{kind}, the idea of Lemma \ref{IC} is simple, but the notation obscures the statement. In words, suppose $A\subset \RHilb_X^{\dim(X)-k+a}$ is a constructible subset such that every element $[Z]\in A$ is actually contained in a member of the family $\mc{X}\to S$. Let $\mc{X}|_s$ be a member of the family. Since $\mc{X}|_s\subset X$ has codimension $e>0$, the tuples of forms $(F_1,\ldots,F_k)$ that vanish on some $Z\subset \mc{X}|_s$ with $[Z]\in A$ actually have zero locus that is $e+a$ dimensional more than expected when restricted to $\mc{X}|_s$, instead of just $a$ dimensional more than expected when regarded on $X$. This will yield a better bound when we apply Lemma \ref{kind} repeatedly. 

The price for that is we need to consider all choices of $s\in S$, so $-\dim(S)$ appears on the right side of \eqref{ICeq}. The role of $c$ and the equality case is just so we won't have to consider the case where $A$ consists of linear spaces separately. If $A$ is the Grassmannian of linear spaces, $\mc{X}\to S$ is the universal family of the Grassmannian and $X=\mb{P}^r$, then the content of Lemma \ref{IC} reduces to the usual incidence correspondence parameterizing choices of $((F_1,\ldots,F_k),\Lambda)$, where $F_1,\ldots,F_k$ are homogenous forms vanishing on $\Lambda$. 
\begin{proof}
To summarize all the objects involved, consider the commutative diagram
\begin{center}
\begin{tikzcd}
 & \Phi^{\mb{P}^r,a+e}_{d_1,\ldots,d_k}(\mc{X}/S) \arrow[r] \arrow[dl]   & \Phi^{\mb{P}^r,a}_{d_1,\ldots,d_k}(X) \arrow[hook,dr]& \\
 S &\wt{\Phi}^{\mb{P}^r,a+e}_{d_1,\ldots,d_k}(\mc{X}/S) \arrow[r] \arrow[u] \arrow[d] \arrow[l] & \wt{\Phi}^{\mb{P}^r,a}_{d_1,\ldots,d_k}(X) \arrow[u] \arrow[d] \arrow[r]& \prod_{i=1}^{k}{W_{r,d_i}}\\
&\RHilb_{\mc{X}/S}^{\dim(X)-k+a}  \arrow[r] \arrow[ul] & \RHilb_X^{\dim(X)-k+a} &
\end{tikzcd}
\end{center}
Restricting to $\mc{A}$ yields
\begin{center}
\begin{tikzcd}
 & \Phi^{\mb{P}^r,a+e}_{d_1,\ldots,d_k}(\mc{X}/S,\mc{A}) \arrow[r, "\phi_1", two heads] \arrow[dl]   & \Phi^{\mb{P}^r,a}_{d_1,\ldots,d_k}(X,A) \arrow[hook,dr]& \\
 S &\wt{\Phi}^{\mb{P}^r,a+e}_{d_1,\ldots,d_k}(\mc{X}/S,\mc{A}) \arrow[r, "\phi_2", two heads] \arrow[u,two heads, "\wt{\pi}_1" ] \arrow[d,two heads, "\wt{\pi}_2"] \arrow[l] & \wt{\Phi}^{\mb{P}^r,a}_{d_1,\ldots,d_k}(X,A) \arrow[u, two heads, "\pi_1"] \arrow[d, two heads, "\pi_2"] \arrow[r]& \prod_{i=1}^{k}{W_{r,d_i}}\\
&\mc{A} \arrow[r, "\phi_3", two heads] \arrow[ul] & A &
\end{tikzcd}
\end{center}
We claim $\phi_1$ is surjective. Indeed, $\wt{\pi}_1$, $\pi_1$ and $\phi_3$ are surjective by definition. The maps $\wt{\pi}_2$ and $\pi_2$ are surjections since the affine spaces $W_{r,d_i}$ contain the zero homogenous form. The map $\phi_2$ is surjective because the fiber of $\phi_2$ over a pair $((F_1,\ldots,F_k),[Z])$ corresponding to a closed point of $\prod_{i=1}^{k}W_{r,d_i}\times \RHilb_X^{\dim(X)-k+a}$ is $\phi_3^{-1}([Z])$, and $\phi_1$ is surjective because each fiber of $\phi_1$ is a union of fibers of $\phi_2$. Indeed,  take a closed point $(F_1,\ldots,F_k)\in \Phi^{\mb{P}^r,a}_{d_1,\ldots,d_k}(X)$. Let $B:=\RHilb_{V(F_1, \ldots, F_k)\cap X}^{\dim(X)-k+a}\cap A\subset A$ be the constructible set of all $[Z]\in A$ such that $Z\subset X\cap V(F_1,\ldots, F_{k})$.  The fiber of $\phi_1$ over $(F_1,\ldots,F_k)$ is $\bigcup_{a\in B}{\phi_3^{-1}(a)}$.

To show the inequality, it suffices to show that $\phi_1$ has fiber dimension at least $c$. This is true as $\bigcup_{a\in B}{\phi_3^{-1}(a)}$ has dimension at least $c$, since each fiber of $\phi_3$ has dimension at least $c$. 

To show the equality case, it suffices to show $\phi_1$ has generic fiber dimension $c$ by the irreducibility of $\Phi^{\mb{P}_S^r,a+e}_{d_1,\ldots,d_k}(\mc{X}/S,\mc{A})$, so we pick  $(F_1,\ldots,F_k)\in \Phi^{\mb{P}^r,a}_{d_1,\ldots,d_k}(X)$ general. Then, $B$ is finite as $\pi_1$ is generically finite, so $\bigcup_{a\in B}{\phi_3^{-1}(a)}$ has dimension exactly $c$ if $\phi_3^{-1}(a)$ has dimension exactly $c$ for all $a\in B$. 
\end{proof}

\section{Partitioning by Span}
In order to apply Lemmas \ref{kind} and \ref{IC}, we need to partition $\RHilb_{\mb{P}^r}$ into constructible sets. We will partition by span.

\begin{Def}
\label{spanD}
Given a positive integer $b$, let $\Span(r,b)\subset \RHilb^{\leq D}_{\mb{P}^r}$ denote the locally closed subset parameterizing geometrically integral schemes $Z\subset \mb{P}^r$ where $Z$ spans a plane of dimension $b$ and have degree at most $D$. 
\end{Def}

We claim $\Span(r,b)$ is locally closed. Indeed, $\bigcup_{i=0}^{b}{\Span(r,i)}$ is closed because upper semicontinuity of dimension implies the locus $\mc{Z}\subset \RHilb_{\mb{P}^r}\times \mb{G}(b,r)$ of pairs $([Z],P)$ with $Z\subset P$ is closed. Also, it is easy to see
\begin{Lem}
We have $\Co(\Span(r,r))=\Span(r,r)$. 
\end{Lem}

\begin{Def}
Define
\begin{align*}
h_{r,a}(d):=(r-a)\binom{d+a-1}{d-1}+\binom{d+a}{d}.
\end{align*}
\end{Def}

By iterating Lemma \ref{kind} and applying Lemma \ref{ARBH}, we get
\begin{Lem}
\label{KIA}
We have $\codim(\Phi^{\mb{P}^r,a}_{d_1,\ldots,d_k}(\mb{P}^r,\Span(r,r)))$ is at least
\begin{align*}
&\min\{\sum_{j=1}^{a}h_{r,r-i_j+j}(d_{i_j}):1\leq i_1<\cdots <i_a\leq k\}.
\end{align*}
Here $a\geq 0$, $k>0$ and $d_i>0$ for each $i$. 
\end{Lem}

\begin{proof}
We will prove this by induction on $k$. If $k=1$, then 
\begin{enumerate}
\item
If $a=0$, then $\codim(\Phi^{\mb{P}^r,a}_{d_1}(\mb{P}^r,\Span(r,r)))=0$, which is at least 0 (the empty sum).
\item
If $a=1$, then $\Phi^{\mb{P}^r,a}_{d_1}(\mb{P}^r,\Span(r,r))$ is just the zero homogenous form in $W_{r,d_1}$, which has codimension 
\begin{align*}
\binom{r+d_1}{d_1}=h_{r,r}(d). 
\end{align*}
\item
If $a>1$, then we are taking the minimum over the empty set, as we cannot choose a subset $S\subset \{1\}$ with cardinality $a$. The minimum of the empty set is $\infty$. This is equal to the codimension of $\Phi^{\mb{P}^r,a}_{d_1}(\mb{P}^r,\Span(r,r))=\emptyset$ by our convention (Definition \ref{DEFDIM}). 
\end{enumerate}
Now, suppose $k>1$. By Lemma \ref{kind}, we know $\codim(\Phi^{\mb{P}^r,a}_{d_1,\ldots,d_k}(X,\Span(r,r)))$ is at least
\begin{align*}
\min\{\codim(\Phi^{\mb{P}^r,a}_{d_1,\ldots,d_{k-1}}(\mb{P}^r,\Span(r,r))),\codim(\Phi^{\mb{P}^r,a-1}_{d_1,\ldots,d_{k-1}}(\mb{P}^r,\Span(r,r)))+h_{\Span(r,r)}(d_k)\}
\end{align*}
By Lemma \ref{ARBH}, $h_{\Span(r,r)}(d_k)\geq h_{r,r-k+a}(d_k)$. By induction, we know $\codim(\Phi^{\mb{P}^r,a}_{d_1,\ldots,d_{k-1}}(\mb{P}^r,\Span(r,r)))$ is at least 
\begin{align*}
\min\{\sum_{j=1}^{a}h_{r,r-i_j+j}(d_{i_j}):1\leq i_1<\cdots<i_a\leq k-1\},
\end{align*}
and $\codim(\Phi^{\mb{P}^r,a-1}_{d_1,\ldots,d_{k-1}}(\mb{P}^r,\Span(r,r)))+h_{r,r-k+a}(d_k)$ is at least 
\begin{align*}
\min\{\sum_{j=1}^{a-1}h_{r,r-i_j+j}(d_{i_j}):1\leq i_1<\cdots<i_{a-1}\leq k-1\}+h_{r,r-k+a}(d_k),
\end{align*}
and taking the minimum over the two yields 
\begin{align*}
\min\{\sum_{j=1}^{a}h_{r,r-i_j+j}(d_{i_j}):1\leq i_1<\cdots<i_a\leq k\}.
\end{align*}
\end{proof}

The bound in Lemma \ref{KIA} depends on the order of the $d_i$'s and it is better to order them so $d_1\leq d_2\leq \cdots\leq d_k$. We could have similarly bounded $\codim(\Phi^{r,a}_{d_1,\ldots,d_k}(\mb{P}^r,\Span(r,b)))$, but this is better dealt with using Lemma \ref{IC}. 
\begin{Lem}
\label{ICA}
We have
\begin{align*}
\codim(\Phi^{\mb{P}^r,a}_{d_1,\ldots,d_k}(\mb{P}^r,\Span(r,b)))&\geq \codim(\Phi^{\mb{P}^b,a+(r-b)}_{d_1,\ldots,d_k}(\mb{P}^b,\Span(b,b)))-\dim(\mb{G}(b,r)),
\end{align*}
and equality holds when $b=a$. 
\end{Lem}

\begin{proof}
We apply Lemma \ref{IC} in the case $S=\mb{G}(b,r)$, $\mc{X}$ is the universal family of $b$-planes over the Grassmannian, $\mc{A}$ is $(\Span(r,b)\times_{K} S)\cap \RHilb_{\mc{X}/S}^{b-k+a,\leq D}$, $Y=\mb{P}^r$, and $c=0$ to get
\begin{align*}
\codim(\Phi^{\mb{P}^r,a}_{d_1,\ldots,d_k}(\mb{P}^r,\Span(r,b)))&\geq \codim(\Phi^{\mb{P}_S^r,a+(r-b)}_{d_1,\ldots,d_k}(\mc{X}/S,\mc{A}))-\dim(\mb{G}(b,r)).
\end{align*}
Let $P\subset \mb{P}^r_K$ correspond to a closed point of $s\in S$. Then, the fiber of $\Phi^{\mb{P}^r_S,a+(r-b)}_{d_1,\ldots,d_k}(\mc{X}/S,\mc{A})\rightarrow S$ over $s$ is isomorphic to
\begin{align*}
\Phi^{\mb{P}^b,a+(r-b)}_{d_1,\ldots,d_k}(\mb{P}^b,\Span(b,b))\times \prod_{i=1}^{k}{W_{r,d_i}/W_{b,d_i}}
\end{align*}
by restricting to $P\cong \mb{P}^b$. 

We now show equality in Lemma \ref{IC} holds when $b=a$. It suffices to show 
\begin{align*}
\wt{\Phi}^{r,a}_{d_1,\ldots,d_k}(\mb{P}^r,\Span(r,a))\rightarrow \Phi^{r,a}_{d_1,\ldots,d_k}(\mb{P}^r,\Span(r,a))
\end{align*}
has generically finite fibers. Therefore, it suffices to show there exists a choice of $(F_1,\ldots,F_k)$ of hypersurfaces such that $V(F_1,\ldots,F_{k})$ contains some $r-k+a$ dimensional plane $P$ and $V(F_1,\ldots,F_{k-a})$ is dimension $r-k+a$, as this means $\bigcap_{i=1}^{k-a}F_i$ cannot contain a positive dimensional family of $r-k+a$ planes.

We will construct $(F_1,\ldots,F_{k-a})$ inductively. Fix a $r-k+a$ dimensional plane $P$, and suppose we have found $F_1,\ldots,F_i$ restricting to zero on $P$ such that $\dim(V(F_1,\ldots,F_i))=r-i$ for $1\leq i<k-a$. Let the components of $V(F_1,\ldots, F_i)$ be $X_1,\ldots,X_{\ell}$. We want to find a homogenous form $F_{i+1}$ that vanishes on $P$ but not on $X_j$ for $1\leq j\leq \ell$. Let $V_j\subset H^{0}(\mb{P}^r,\mathscr{O}_{\mb{P}^r}(d_{i+1}))$ be the vector space of degree $d_{i+1}$ forms vanishing on $X_j$ and let $V\subset  H^{0}(\mb{P}^r,\mathscr{O}_{\mb{P}^r}(d_{i+1}))$ be the vector space of degree $d_{i+1}$ forms vanishing on $P$. 

Let $h$ denote the Hilbert function. Since $\dim(X_j)>\dim(P)$ and $P$ is a linear space, $h_{P}(d_{i+1})< h_{X_j}(d_{i+1})$, so the inclusion $V_j\subset V$ is strict. Therefore, $V\backslash \bigcup_{j=1}^{\ell}{V_j}$ is nonempty, and we can pick any $F_{i+1}$ in $V\backslash \bigcup_{j=1}^{\ell}{V_j}$. Finally, we pick $F_{k-a+1},\ldots,F_k$ to be any homogenous forms restricting to zero on $P$.  
\end{proof}

\begin{Def}
Given $r,a$ and degrees $d_1,\ldots,d_k$, let 
\begin{align*}
F_{r,a}(d_1,\ldots,d_k)&:=\min\{\sum_{j=1}^{a}h_{r,r-i_j+j}(d_{i_j}):1\leq i_1<\cdots<i_a\leq k\}
\end{align*}
\end{Def}

\begin{Def}
Given $r,a$, degrees $d_1,\ldots,d_k$ and $r-k+a\leq b\leq r$, let
\begin{align*}
G_{r,a,b}(d_1,\ldots,d_k)&:= F_{b,a+(r-b)}(d_1,\ldots,d_k)-\dim(\mb{G}(b,r)).
\end{align*}
\end{Def}

Combining Lemmas \ref{KIA} and \ref{ICA}, we have
\begin{Thm}
\label{FML}
For $d_1\leq d_2\leq \cdots\leq d_k$, we have
\begin{align*}
\codim(\Phi^{\mb{P}^r,a}_{d_1,\ldots,d_k}(\mb{P}^r, \Span(r-k+a)))&=G_{r,a,r-k+a}(d_1,\ldots,d_k)\\
&=-(r-k+a+1)(k-a)+\sum_{i=1}^{k}\binom{d_i+r-k+a}{r-k+a}.
\end{align*}
and
\begin{align*}
\codim(\Phi^{\mb{P}^r,a}_{d_1,\ldots,d_k}(\mb{P}^r)\backslash \Phi^{\mb{P}^r,a}_{d_1,\ldots,d_k}(\mb{P}^r,\Span(r-k+a)))&\geq \min_{b=r-k+a+1}^{r}{G_{r,a,b}(d_1,\ldots,d_k)}
\end{align*}
\end{Thm}

\subsection{Hypersurfaces containing a curve}
We specialize Theorem \ref{FML} to the case of hypersurfaces containing some curve in order to get cleaner results. At this point, we have reduced our problem to combinatorics. 

What really matters in our case is not the exact codimension, but the difference between the codimension of the $k$-tuples of homogenous forms all vanishing on the same line, and the codimension of the $k$-tuples of homogenous forms all vanishing on some curve other than a line. 
\begin{Def}
\label{HrabD}
Let $H_{r,a,b}(d_1,\ldots,d_k):=G_{r,a,b}(d_1,\ldots,d_k)-G_{r,a,r-k+a}(d_1,\ldots,d_k)$. 
\end{Def}

Our goal is to prove
\begin{Thm}
\label{slope}
Suppose $r-k+a=1$ and $2\leq d_1\leq d_2\leq \cdots\leq d_k$. Then, if $d_i\leq d_1+(i-1)\binom{d_1}{2}$, then 
\begin{align*}
\codim(\Phi^{r,a}_{d_1,\ldots,d_k}(\mb{P}^r)\backslash \Phi^{r,a}_{d_1,\ldots,d_k}(\mb{P}^r,\Span(r-k+a)))-\codim(\Phi^{r,a}_{d_1,\ldots,d_k}(\mb{P}^r))
\end{align*}
is at least
\begin{align*}  
\begin{cases}
r-1 &\text{ if $k=r$}\\
r-1 + (d_1-2)(r-2)+d_1(k-r)& \text{ if $k>r$}.
\end{cases}
\end{align*}
\end{Thm}

The reason why applying Theorem \ref{FML} is not completely straightforward is because the definition of $F_{r,a}$ requires us to take a minimum over a large choice of indices. 
\begin{Def}
We say that the minimum for $F_{r,a}(d_1,\ldots,d_k)$ is achieved at the indices $1\leq  i_1<\cdots<i_a \leq k$ if 
\begin{align*}
F_{r,a}(d_1,\ldots,d_k)&=\sum_{j=1}^{a}h_{r,r-i_j+j}(d_{i_j})
\end{align*}
Note that the choice of $i_1<\cdots<i_k$ is not necessarily unique.
\end{Def}

\begin{Def}
Similarly, we say that the minimum for $G_{r,a,b}(d_1,\ldots,d_k)$ is achieved at  $1\leq  i_1<\cdots<i_{a+r-b}\leq k$ if the minimum for $F_{b,a+(r-b)}(d_1,\ldots,d_k)$ is achieved at $i_1<\cdots < i_{a+r-b}$.
\end{Def}

From directly computing, we have the following easy facts.
\begin{Lem}
\label{US2}
We have
\begin{align*}
h_{r,a}(d)-h_{r,a}(d-1)&=(r-a)\binom{d+a-2}{d-1}+\binom{d+a-1}{d}.
\end{align*}
\end{Lem}


\begin{Lem}
\label{US1}
We have
\begin{align*}
h_{r,a+1}(d)-h_{r,a}(d)&= (r-a)\binom{d+a-1}{a+1}.
\end{align*}
\end{Lem}


Lemma \ref{KCL} is the key combinatorial lemma in this section. 
\begin{Lem}
\label{KCL}
For $r-k+a=1$, we have
\begin{align*}
\min\{H_{r,a,b}(d_1,\ldots,d_k):\ d=d_1\leq d_2\leq \cdots\leq d_k, d_i\leq d+(i-1)\binom{d}{2}\}=H_{r,a,b}(d,\ldots,d)
\end{align*}
\end{Lem}

\begin{proof}
Suppose we are given $d= d_1\leq d_2\leq \cdots\leq d_k$. We want to show
\begin{align*}
H_{r,a,b}(d_1,\ldots,d_k)&\geq H_{r,a,b}(d,\ldots,d).
\end{align*}
We will use induction on $d_1+\cdots+d_k$, so it suffices to find $d_1',\ldots,d_k'$ with $d_1'+\cdots+d_k'$ is less than $d_1+\cdots+d_k$ such that $H_{r,a,b}(d_1',\ldots,d_k')\leq H_{r,a,b}(d_1,\ldots,d_k)$. The proof is structured like an induction, but perhaps it is intuitively easier to think of it as taking $(d_1,\ldots,d_k)$, and altering the degrees bit by bit until they reach $(d,\ldots,d)$, while all the time not increasing the value of $H_{r,a,b}(d_1,\ldots,d_k)$. 

Suppose the minimum of $G_{r,a,b}(d_1,\ldots,d_k)$ is achieved at $i_1<\cdots<i_{a-r+b}$. Suppose $d_{i_1}>d$. 
Define
\begin{align*}
d_i'&=
\begin{cases}
\max\{d,d_i-1\} &\text{if $i\leq i_1$}\\
d_i & \text{if $i>i_1$}.
\end{cases}
\end{align*}
We claim that $H_{r,a,b}(d_1,\ldots,d_k)\geq H_{r,a,b}(d_1',\ldots,d_k')$. To see this,
\begin{align*}
H_{r,a,b}(d_1,\ldots,d_k)-H_{r,a,b}(d_1',\ldots,d_k')&=\\
G_{r,a,b}(d_1,\ldots,d_k)-G_{r,a,b}(d_1',\ldots,d_k')-G_{r,a,1}(d_1,\ldots,d_k)+G_{r,a,1}(d_1',\ldots,d_k')&\geq\\
h_{b,b-i_1+1}(d_{i_1})-h_{b,b-i_1+1}(d_{i_1}-1)-i_1&\geq\\
(i_1-1)\binom{d_{i_1}+b-i_i-1}{d_{i_1}-1}+\binom{d_{i_1}+b-i_i}{d_{i_1}}-i_1&,
\end{align*}
where we applied Lemma \ref{US2} in the last step. Since $r-k+a=1$, $i_1\leq b$. Therefore, we see the quantity above is at least $(i_1-1)+1-i_1=0$. Therefore, if $d_{i_1}>d$, we are done by induction. Otherwise, we can assume $d=d_1=\cdots=d_{i_1}$. 

Suppose $j$ is the minimum index for which $i_j-i_{j-1}>1$. If there is no such index, let $j=a+r-b+1$. First, we reduce to the case where $d_{i_1}=\cdots=d_{i_{j-1}}=d$. Let $1\leq \ell< j$ be the minimum index such that $d_{i_\ell}>d$. We can use the same trick as before. If we again define
\begin{align*}
d_i'&=
\begin{cases}
d & \text{if $i<i_\ell$}\\
d_{i_\ell}-1&\text{if $i=i_\ell$}\\
d_i & \text{if $i>i_\ell$},
\end{cases}
\end{align*}
then we again see that $H_{r,a,b}(d_1,\ldots,d_k)\geq H_{r,a,b}(d_1',\ldots,d_k')$, as
\begin{align*}
H_{r,a,b}(d_1,\ldots,d_k)-H_{r,a,b}(d_1',\ldots,d_k')&=\\
G_{r,a,b}(d_1,\ldots,d_k)-G_{r,a,b}(d_1',\ldots,d_k')-G_{r,a,1}(d_1,\ldots,d_k)+G_{r,a,1}(d_1',\ldots,d_k')&\geq\\
h_{b,b-i_\ell+\ell}(d_{i_{\ell}})-h_{b,b-i_\ell+\ell}(d_{i_{\ell}}-1)-1&=\\
(i_\ell-\ell)\binom{d_{i_\ell}+b-i_\ell+\ell-2}{d_{i_\ell}-1}+\binom{d_{i_\ell}+b-i_\ell+\ell-1}{d_{i_\ell}}-1&\geq 0.
\end{align*}
Again, we are using $i_\ell-\ell\leq b-1$. Therefore, we see that if $d_{i_\ell}>d$, then we are done by induction. Otherwise, we can now assume $d_{i_1}=\cdots=d_{i_{j-1}}=d$. Recall from above, we also have by assumption that $d_i=d$ for $i\leq i_{j-1}$, $i_2=i_1+1,\ldots,i_{j-1}=i_{j-2}+1$. 

If $i_1=b$, then $i_1,\ldots,i_{a+(r-b)}$ is precisely $b,\ldots,k$, so $d_i=d$ for all $1\leq i\leq k$, in which case we are done. Therefore, we can assume $i_1<b$. 

Suppose $i_1<b$, so in particular $j\leq a$ and $i_{j}>i_{j-1}+1$. Note $d_{i_{j-1}+1}>d$, because otherwise replacing $i_{j-1}$ by $i_{j-1}+1$ would decrease $\sum_{\ell=1}^{a+(r-b)}{h_{b,r-i_\ell+\ell}(d_{i_\ell})}$, contradicting the assumption that the minimum of $G_{r,a,b}(d_1,\ldots,d_k)$ is achieved at $i_1<\cdots<i_{a+(r-b)}$. Let 
\begin{align*}
d_i'&=
\begin{cases}
d & \text{if $i\leq i_{j-1}+1$}\\
d_i & \text{if $i>i_{j-1}+1$}.
\end{cases}
\end{align*}
We claim $H_{r,a,b}(d_1,\ldots,d_k)\geq H_{r,a,b}(d_1',\ldots,d_k')$. To see this, let
\begin{align*}
i_\ell' &=
\begin{cases}
i_\ell+1 &\text{if $\ell<j$}\\
i_\ell &\text{if $\ell\geq j$}
\end{cases}
\end{align*}
we see
\begin{align*}
H_{r,a,b}(d_1,\ldots,d_k)-H_{r,a,b}(d_1',\ldots,d_k')&=\\
G_{r,a,b}(d_1,\ldots,d_k)-G_{r,a,b}(d_1',\ldots,d_k')-G_{r,a,1}(d_1,\ldots,d_k)+G_{r,a,1}(d_1',\ldots,d_k')&\geq\\
\sum_{\ell=1}^{a+(r-b)}{h_{b,r-i_\ell+\ell}(d_{i_\ell})}-\sum_{\ell=1}^{a+(r-b)}{h_{b,r-i_\ell'+\ell}(d_{i_\ell}')}-(d_{i_{j-1}+1}-d)&\geq\\
(j-1)(h_{b,b-i_1+1}(d)-h_{b,b-i_1}(d))-\binom{d}{2}i_{j-1}&=\\
(j-1)i_1\binom{d+b-i_1-1}{b-i_1+1}-\binom{d}{2}(i_1+j-2)&.
\end{align*}
Since $i_1<b$, we know $b-i_1\geq 1$. So, this is at least 
\begin{align*}
(j-1)i_1\binom{d}{2}-\binom{d}{2}(i_1+j-2)=\binom{d}{2}((j-1)i_1-(i_1+j-2)).
\end{align*}
Since $i_1\geq 1$ and $j>1$, $(j-1)i_1-(i_1+j-2)\geq 0$. Therefore, $H_{r,a,b}(d_1,\ldots,d_k)$ is at least $H_{r,a,b}(d_1',\ldots,d_k')$, so we are again done by induction. 
\end{proof}

Now we prove Theorem \ref{slope}. 
\begin{proof}
From Lemma \ref{KCL}, it suffices to consider the case $d_1=\cdots=d_k=d$. Then, we see
\begin{align*}
G_{r,a,b}(d_1,\ldots,d_k) &= -\dim(\mb{G}(b,r))+(a+(r-b))(db+1).
\end{align*}
This is a quadratic in $b$ with leading coefficient $1-d<0$, so it suffices to show 
\begin{align*}
\begin{tabular}{c}
$H_{r,a,2}(d_1,\ldots,d_k),$\\
$H_{r,a,r}(d_1,\ldots,d_k)$
\end{tabular}
\geq 
\begin{cases}
r-1 &\text{ if $k=r$}\\
r-1 + (d-2)(r-2)+d(k-r)& \text{ if $k>r$}.
\end{cases} 
\end{align*}
We see
\begin{align}
H_{r,a,2}(d_1,\ldots,d_k)&=-3(r-2)+(k-1)(2d+1)-k(d+1)+2(r-1)\nonumber\\
&=d (k - 2) - r + 3=r-1+d(k-2)-(2r-4)\nonumber\\
&=r-1+(d-2)(r-2)+d(k-r).\label{Hra2}
\end{align}
and
\begin{align}
H_{r,a,r}(d_1,\ldots,d_k) &= (k-r+1)(rd+1)-k(d+1)+2(r-1)\nonumber\\
&=d (k r - k - r^2 + r) + r - 1=d(k-r)(r-1)+(r-1).\label{Hrar}
\end{align}
To finish, we need to check $H_{r,a,2}(d_1,\ldots,d_k)\geq H_{r,a,r}(d_1,\ldots,d_k)$ for $k=r$ and that $H_{r,a,2}(d_1,\ldots,d_k)\leq H_{r,a,r}(d_1,\ldots,d_k)$ for $k>r$. We calculate
\begin{align}
H_{r,a,r}(d_1,\ldots,d_k)-H_{r,a,2}(d_1,\ldots,d_k)&=d(k-r)(r-2)-(d-2)(r-2)\nonumber\\
&=(r-2)(d(k-r)-(d-2)), \nonumber 
\end{align}
and $d(k-r)-(d-2)$ is positive if $k>r$ and nonnegative if $k=r$ and $d\geq 2$. 
\end{proof}

\section{Application: Lines in a hypersurface through a point}
Let $\mc{F}_{r,d}\rightarrow W_{r,d}$ be the universal hypersurface. Let $\F(\mc{F}_{r,d}/W_{r,d})$ denote the lines on the universal hypersurface, or the relative Hilbert scheme of lines in the family $\mc{F}_{r,d}\rightarrow W_{r,d}$. The universal family over $\F(\mc{F}_{r,d}/W_{r,d})$ is a $\mb{P}^1$-bundle $\Fo(\mc{F}_{r,d}/W_{r,d})\rightarrow \F(\mc{F}_{r,d}/W_{r,d})$ corresponding to a choice of a line and a point on that line. 

There is an evaluation map $\Fo(\mc{F}_{r,d}/W_{r,d})\rightarrow \mc{F}_{r,d}$, and the expected fiber dimension is $r-1-d$. We are interested in when the fiber dimension jumps. In the statement of Theorem \ref{EVM} below, we will need to refer to \emph{Eckardt points}. In the notation of \cite{Eckardt}, the $0$-Eckardt points of a smooth variety $X\subset\mb{P}^r$ are the points for which the second fundamental form at $x\in X$ vanishes. 

More concretely, if $X=V(F)$ for a nonzero homogenous form $F$ of degree $d$ and $x$ is the origin in an affine chart of $\mb{P}^n$, and we expand $F$ around $x$ as $F=F_1+\cdots+F_d$, where $F_i$ is the degree $i$ part of $F$ after dehomogenization, then $x$ is an Eckardt point if and only if $F_1$ divides $F_2$. Therefore, we make the following definition
\begin{Def}
Let $F(X_0,\ldots,X_r)$ be a homogenous form of degree $d$ that vanishes on $p=[0:\cdots:0:1]\in \mb{P}^r$, then $p$ is an Eckardt point of $F$ if 
\begin{align*}
\frac{1}{X_r^d}F(X_0,\ldots,X_r)&= 0+F_1(\frac{X_0}{X_r},\ldots,\frac{X_{r-1}}{X_r})+\cdots+F_d(\frac{X_0}{X_r},\ldots,\frac{X_{r-1}}{X_r}),
\end{align*}
where $F_i$ is homogenous of degree $i$. For a general choice of $p\in V(F)$, we take a $PGL_{r+1}$ translate $\phi: \mb{P}^r\to \mb{P}^r$ where $\phi([0:\cdots:0:1])=p$ and apply the definition above to $\phi^{*}F$. 
\end{Def}

In particular, every point is an Eckardt point of the zero form. Similarly, we say that a homogenous form $F$ on $\mb{P}^r$ is smooth if $V(F,\partial_{X_0}F,\ldots,\partial_{X_r}F)$ is empty. 

\begin{Thm}
\label{EVM}
Let $U\subset W_{r,d}$ be the open subset of smooth homogenous forms $F$ of degree $d\geq 3$ in $\mb{P}^r$ for $r\geq 2$. Let $Z\subset U$ be the closed subset of homogenous forms $F$ for which the evaluation map $\Fo(X)\rightarrow X$ has a fiber of dimension greater than $r-1-d$, where $X=V(F)$. 

Then, $Z$ has a unique component of maximum dimension, except in the case $d=4$, $r=5$. The component(s) of $Z$ of maximal dimension are as follows:
\begin{enumerate}[(1)]
\item
the forms $F$ with an Eckardt point for $d\leq r-2$ or $d=r-1$ and $r\leq 5$
\item
the forms $F$ vanishing on  a 2-plane for $d=r-1$ and $r\geq 5$
\item
the forms $F$ vanishing on a line for $d\geq r$.
\end{enumerate}
\end{Thm}

\subsection{Application of filtration by span}
To prove Theorem \ref{EVM}, we first need to understand what happens on the universal hypersurface, which will be given by Proposition \ref{UHS} in this subsection. 
\begin{Prop}
\label{2r1}
The unique component of $\Phi^{\mb{P}^r,1}_{2,\ldots,r+1}(\mb{P}^r)$ of maximal dimension is the component parameterizing $r$-tuples of hypersurfaces whose common vanishing locus contains a line, and that component is of codimension $\frac{r^2+r+4}{2}$.
For $r\geq 2$, the unique component of second largest dimension of $\Phi^{\mb{P}^r,1}_{2,\ldots,r+1}(\mb{P}^r)$ is the locus of tuples of hypersurfaces $(F_1,\ldots,F_r)$ where $F_1$ is identically zero, which has codimension $\binom{r+2}{2}$. 
\end{Prop}

\begin{proof}
The component of $\Phi^{\mb{P}^r,1}_{2,\ldots,r+1}(\mb{P}^r)$ of maximal dimension can be identified by applying Theorem \ref{slope}. To find the second largest component in terms of dimensions, we see that Theorem \ref{slope} also says the difference between the dimensions of the largest and second largest components of $\Phi^{\mb{P}^r,1}_{2,\ldots,r+1}(\mb{P}^r)$ is at least $r-1$. Since 
\begin{align*}
\binom{r+2}{2}-(\frac{r(r+5)}{2}-2(r-1))=r-1,
\end{align*}
we see that the locus of tuples of hypersurfaces $(F_1,\ldots,F_r)$ where $F_1$ is identically zero is a component of second largest dimension. To finish, we need to show uniqueness. Let $Z\subset \Phi^{r,1}_{2,\ldots,r+1}(\mb{P}^r)$ be a component of second highest dimension. From the proof of Theorem \ref{slope}, we know that if 
\begin{align*}
\dim(Z) = \dim(Z\cap \Phi^{\mb{P}^r,1}_{2,\ldots,r+1}(\mb{P}^r,\Span(b))),
\end{align*}
then $b=2$ or $b=r$. More precisely, Lemma \ref{ICA} says
\begin{align*}
\codim(\Phi^{\mb{P}^r,1}_{2,2,\ldots,2}(\mb{P}^r,\Span(b)))&\geq G_{r,1,b}(2,2,\ldots,2)\\
\codim(\Phi^{\mb{P}^r,1}_{2,3,\ldots,r+1}(\mb{P}^r,\Span(b)))&\geq G_{r,1,b}(2,3,\ldots,r+1)\\
\codim(\Phi^{\mb{P}^r,1}_{2,2,\ldots,2}(\mb{P}^r,\Span(1)))&= G_{r,1,1}(2,2,\ldots,2)\\
\codim(\Phi^{\mb{P}^r,1}_{2,3,\ldots,r+1}(\mb{P}^r,\Span(1)))&= G_{r,1,1}(2,3,\ldots,r+1)\\
\end{align*}
Lemma \ref{KCL} says that
\begin{align*}
H_{r,1,b}(2,3,\ldots,r+1)\geq H_{r,1,b}(2,2,\ldots,2),
\end{align*}
so by definition of $H_{r,1,b}$ (Definition \ref{HrabD})
\begin{align*}
G_{r,1,b}(2,3,\ldots,r+1)-G_{r,1,1}(2,3,\ldots,r+1)&\geq G_{r,1,b}(2,2,\ldots,2)-G_{r,1,1}(2,2,\ldots,2).
\end{align*}
Since
\begin{align*}
G_{r,1,b}(2,2,\ldots,2) &= -\dim(\mb{G}(b,r))+(1+(r-b))(2b+1)
\end{align*}
is quadratic in $b$ with negative leading coefficient, we see that $G_{r,1,b}(2,2,\ldots,2)-G_{r,1,1}(2,2,\ldots,2)=H_{r,1,b}(2,2,\ldots,2)$ is minimized over $b\in \{2,\ldots,r\}$ when $b=2$ or $b=r$. Also, \eqref{Hra2} and \eqref{Hrar} from the proof of Theorem \ref{slope} yields
\begin{align*}
H_{r,1,2}(2,\ldots,2)=H_{r,1,r}(2,\ldots,2)=r-1. 
\end{align*}
Therefore, the only two values of $b$ for which $G_{r,1,b}(2,3,\ldots,r+1)$ would possible be equal to $r-1$ is when $b=2$ or $b=r$. 

If $r>2$, we can rule out the case $b=2$ directly, as Lemma \ref{ICA} implies $\codim( \Phi^{r,1}_{2,\ldots,r+1}(\mb{P}^r,\Span(b)))\geq G_{r,r-1,2}(2,\ldots,r+1)$, and it is easy to check $G_{r,r-1,2}(2,\ldots,r+1)$ achieves its minimum at the choice of indices $1<3<4<\cdots<r$, so 
\begin{align*}
\codim( \Phi^{\mb{P}^r,1}_{2,\ldots,r+1}(\mb{P}^r,\Span(b)))\geq 6+ \sum_{i=3}^{r}(2i+1)=r^2+2r-2,
\end{align*}
which is greater than $\binom{r+2}{2}$ for $r>2$. To finish, it suffices to show that 
\begin{align*}
h_{r,r-i+1}(1+i)=(i-1)\binom{r+1}{i}+\binom{r+2}{i+1}
\end{align*}
 achieves its unique minimum at $i=1$ over $1\leq i\leq r$. By looking at the second term, we see that it suffices to compare the cases when $i=1$ and $i=r$, and we see
\begin{align*}
(r-1)(r+1)+(r+2)-\binom{r+2}{2}=\frac{1}{2}(r^2-r)
\end{align*}
which is greater than zero when $r>1$. 
\end{proof}

Proposition \ref{2r} is an example where we apply Theorem \ref{FML} to a case where the hypersurfaces all contain the same surface rather than a curve. This will be easier than Proposition \ref{2r1}, as we can make cruder approximations. 
\begin{Prop}
\label{2r}
The unique component of $\Phi^{\mb{P}^r,1}_{2,\ldots,r}(\mb{P}^r)$ of largest dimension is the locus of tuples $(F_1,\ldots,F_{r-1})$ of degrees $(2,\ldots, r)$ such that $F_1=0$. 
\end{Prop}

\begin{proof}
Applying Lemma \ref{E1} and Lemma \ref{ICA}, we find
\begin{align*}
\codim(\Phi^{\mb{P}^r,1}_{2,\ldots,r}(\mb{P}^r,\Span(b)))&\geq \sum_{i=b+2}^{r+2}{\binom{i}{2}}-(b+1)(r-b)\\
&\geq \binom{r+3}{3}-\binom{b+2}{3}-(b+1)(r-b),
\end{align*}
and equality holds when $b=r$. Let $A(b,r):=\binom{r+3}{3}-\binom{b+2}{3}-(b+1)(r-b)$. Taking the difference $A(r,b)-A(1,b)$, we get
\begin{align*}
-\frac{1}{6} (b - r) (b^2 + b r - 3 b + r^2 + 3 r - 4).
\end{align*}
Since $b^2 + b r - 3 b + r^2 + 3 r - 4>0$ for $b<r$, we only have to deal with the case $b=r$. In this case, we see that $h_{r,r-i+1}(1+i)$  achieves its unique minimum at $i=1$ over $1\leq i\leq r-1$, as in the proof of Proposition \ref{2r1}. 
\end{proof}

\begin{Lem}
\label{E1}
We have
\begin{align*}
\codim(\Phi^{\mb{P}^r,a}_{2,\ldots,r+a-1}(\mb{P}^r,\Span(r)))\geq \binom{r+2}{2}+\binom{r+3}{2}+\cdots+\binom{r+a+1}{2}.
\end{align*}
\end{Lem}

\begin{proof}
We want to apply Lemma \ref{KIA}. However, instead of trying to determine the minimum, we can crudely approximate
\begin{align*}
F_{r,a}(2,\ldots,r+a-1) &=\min\{\sum_{j=1}^{a}(i_j-j)\binom{(i_j+1)+(r-i_j+j)-1}{i_j}+\binom{(i_j+1)+(r-i_j+j)}{i_j+1}\\&:S=\{i_1,i_2,\cdots,i_a\}\subset \{1,\ldots,r+a-2\},i_1<\cdots<i_a\}&
\end{align*}
by
\begin{align*}
&\min\{\sum_{j=1}^{a}\binom{(i_j+1)+(r-i_j+j)}{i_j+1}:S=\{i_1,i_2,\cdots,i_a\}\subset \{1,\ldots,r+a-2\},i_1<\cdots<i_a\}&=\\
&\min\{\sum_{j=1}^{a}\binom{1+r+j}{i_j+1}:S=\{i_1,i_2,\cdots,i_a\}\subset \{1,\ldots,r+a-2\},i_1<\cdots<i_a\}&.
\end{align*}
Since $i_j\leq (r+a-2)-(a-j)=r+j-2$, $\binom{1+r+j}{i_j+1}\leq \binom{1+r+j}{2}$. Therefore, the sum is bounded below by
\begin{align*}
\binom{r+2}{2}+\binom{r+3}{2}+\cdots+\binom{r+a+1}{2}.
\end{align*}
\end{proof}

\begin{Prop}
\label{UHS}
Suppose $r\geq d+1$ and $d\geq 2$. Let $\mc{F}_{r,d}\rightarrow W_{r,d}$ be the universal hypersurface, where $W_{r,d}\cong \mb{A}^{\binom{r+d}{r}}$ parameterizes hypersurfaces of degree $d$ in $\mb{P}^r$. Let $\mc{Z}\subset \mc{F}_{r,d}$ be the locus where the fiber of $\Fo(\mc{F}_{r,d}/W_{r,d})\rightarrow F_{r,d}$ has dimension greater than $r-1-d$. The unique component of largest dimension of $\mc{Z}$ is the locus of points $(X,p)\in \mc{F}_{r,d}\subset W_{r,d}\times\mb{P}^r$ of $\mc{F}_{r,d}$ where $p$ is a singular point of $X$. More importantly, 
\begin{enumerate}
\item
for $d<r-1$, the unique component of second largest dimension is the points $(X,p)$ where $p$ is an Eckardt point of $X$
\item
for $d=r-1$, the unique components of second and third largest components are
\begin{enumerate}
\item
the points $(X,p)$ where $X$ contains a plane through $p$
\item
the points $(X,p)$ where $p$ is an Eckardt point of $X$ 
\end{enumerate}
\end{enumerate}
\end{Prop}

\begin{proof}
Fix $p\in \mb{P}^r$ and note the fiber over $p$ of the projection $\pi: \mc{F}_{r,d}\rightarrow \mb{P}^r$ is a hyperplane $\pi^{-1}(p)\subset W_{r,d}$. Let $\pi^{-1}(p)\cap \mc{Z}$ be $Z_p$. We see $Z_p\subset W_{r,d}$ is some closed subset. Given a hypersurface $X\subset\mb{P}^r$ of degree $d$ through $p$ given by a homogeous polynomial $F$ of degree $d$, we can take an affine chart $p\in \mb{A}^n$ where $p$ is the origin, and expand $F=F_1+\cdots+F_d$ around $p$. Here, $F_i$ is the degree $i$ part of $F$ once we restrict to $\mb{A}^n$. Lines through $p$ in $\mb{P}^r$ are parameterized by $\mb{P}^{r-1}$, and the lines through $p$ in $X$ are given by $\{F_1=\cdots=F_d=0\}$ in $\mb{P}^{r-1}$. See the proof of Theorem 2.1 in \cite{LowDegree} for more details and an approach that behaves better as we vary $p$. 

Since specifying the Taylor expansion $(F_1,\ldots,F_d)$ of $F$ around a point $p$ is equivalent to specifying $F$, $Z_p\cong \Phi^{r-1,1}_{1,\ldots,d}(\mb{P}^{r-1})$. The locus where $F_1$ is identically zero corresponds to a choice of hypersurface $X$ through $p$ that is singular at $p$, and this happens in codimension $r$. If we assume $F_1$ is not zero, then we want to restrict to the hyperplane cut out by $F_1$. Take the open subset $U\subset \Phi^{r-1,1}_{1,\ldots,d}(\mb{P}^{r-1})$ of tuples $(F_1,\ldots,F_d)$ where $F_1\neq 0$. There is a map $U\rightarrow (\mb{P}^{r-1})^{*}$ given by $(F_1,\ldots,F_d)$ mapping to $F_1$. Each fiber is isomorphic to 
\begin{align*}
\Phi^{r-2,1}_{2,\ldots,d}(\mb{P}^{r-1})\times \prod_{i=2}^{d}{W_{r-1,i}/W_{r-2,i}}.
\end{align*}
If $d<r-1$, by Proposition \ref{2r}, we find the unique componentof largest component of $\Phi^{r-2,1}_{2,\ldots,d}(\mb{P}^{r-1})$ is when the quadric is identically zero, which corresponds to when $F_2$ restricted to $F_1$ is zero. Equivalently, $p$ being an Eckardt point of $X$. 

If $d=r-1$, by Proposition \ref{2r1}, we find the unique component of largest dimension of $\Phi^{r-2,1}_{2,\ldots,d}(\mb{P}^{r-1})$ is when $(F_2,\ldots,F_d)$ all contain a line lying in $F_1$, which is equivalent to $X$ containing a plane through $p$. By Proposition \ref{2r1}, the unique component of second largest dimension of $\Phi^{r-2,1}_{2,\ldots,d}(\mb{P}^{r-1})$ is when the quadric is zero, which corresponds to the case where $p$ is an Eckardt point of $X$. 
\end{proof}

\subsection{Facts about general hypersurfaces}
To derive Theorem \ref{EVM} from Proposition \ref{UHS}, we require facts about hypersurfaces that are tedious but easy to check. In characteristic 0, many of these statements are immediate, as  smooth hypersurfaces all have finitely many Eckardt points (see the discussion under Corollary 2.2 in \cite{cubic}) and the Fermat hypersurface $\{X_0^d+\cdots+X_r^d=0\}\subset \mb{P}^r$ contains Eckardt points and planes, and is smooth when the characteristic does not divide $d$. 
\begin{Lem}
\label{BP1}
The following hold independently of the characteristic of our algebraically closed base field $K$. 
\begin{enumerate}
\item
There exists a smooth hypersurface of degree $d>1$ in $\mb{P}^r$ containing a 2-plane if and only if $r\geq 5$. 
\item
For $r\geq 2$ and $d\geq 2$, there exists a smooth hypersurface $X$ of degree $d$ in $\mb{P}^r$ and a hyperplane $H$ such that $X\cap H$ is a cone in $H\cong \mb{P}^{r-1}$. 
\end{enumerate}
\end{Lem}
\begin{proof}
We know that if $X\subset\mb{P}^r$ is a smooth hypersurface of degree $d>1$ contains a linear space $\Lambda$ of dimension $m$, then $r\geq 2m+1$, for example from Proposition 1 in the Appendix of \cite{APP}. 

To prove (2), Let $V\subset W_{r,d}$ be the linear subspace of forms whose expansion around $[1:0:\cdots:0]$ in the affine chart $X_{0}\neq 0$ is of the form $f_0+f_1+\cdots+f_d$, where $f_0=0$ and $x_1| f_i$ for $i=1,2,\ldots,d-1$. Here, $x_1:=\frac{X_1}{X_0}$ is one of the coordinates after dehomogenization. Let $\mc{X}\subset V\times \mb{P}^r$ be the incidence correspondence of pairs $(F,p)$, where $F\in V$ and $p\in \{F=0\}$ is a singular point. Here, our convention is that the zero homogenous form is singular at every point.

Consider the fiber of the map $\pi: \mc{X}\rightarrow \mb{P}^r$. If $p\notin \{X_1=0\}$, then we can assume that $p=[0:1:0:\cdots:0]$. We want to check that the $r+1$ conditions being singular at $p$ imposes on $W_{r,d}$ also imposes $r+1$ conditions on $V$. If we let $I$ denote a multi-index, a general element in $W_{r,d}$ can be written as $\sum_{I}c_I X^I$ and $V$ is cut out by the conditions that $c_I=0$ for all monomials $X^I$ divisible by $X_0$ but not $X_1$. Being singular at $p$ imposes the conditions that $c_I=0$ for $X_1^{d-1}$ dividing $X^I$. Since $d>1$, these impose $r+1$ independent conditions on $V$.

Suppose now $p\in \{X_1=0\}$, but $p\neq [1:0:\cdots:0]$, then we can assume $p=[0:\cdots:0:1]$, in which case being singular at $p$ imposes $r$ conditions on $V$. If $p=[1:0:\cdots:0]$, then being singular at $p$ imposes 1 condition on $V$. Combining the three cases, we see that $\dim(\mc{X})=\max\{\dim(V)-1,\dim(V)-1, \dim(V)-1\}$, so the projection $\mc{X}\rightarrow V$ cannot be surjective. 
\end{proof}

\begin{Lem}
\label{BP2}
The following hold independently of the characteristic of our algebraically closed base field $K$. 
\begin{enumerate}
\item
If $\binom{d+2}{2}>3(r-2)$, then a general hypersurface $X\subset\mb{P}^r$ of degree $d$ does not contain a 2-plane, and a general hypersurface containing a 2-plane contains exactly one 2-plane. 
\item
If $d\geq 3$ and $r\geq 3$, then a general hypersurface $X\subset\mb{P}^r$ of degree $d$ containing an Eckardt point contains only one Eckardt point. 
\end{enumerate}
\end{Lem}

\begin{proof}
The proof strategy is similar to the proof of Lemma \ref{BP1}. For example, suppose we wanted to prove (2). We can consider the incidence correspondence $\mc{I}\subset W_{r,d}\times(\mb{P}^r)^{*}\times \mb{P}^r$ consisting of triples $(F,H,p)$ such that $p\in H$ and $F$ restricted to $H$ vanishes at $p$ up to third order. By considering the projection to $(\mb{P}^r)^{*}\times \mb{P}^r$, we see
\begin{align*}
\dim(\mc{I})=\dim(W_{r,d})-\binom{r+1}{2}+(2r-1).
\end{align*}
We can also consider the incidence correspondence $\mc{J}\subset W_{r,d}\times(\mb{P}^r)^{*}\times \mb{P}^r\times (\mb{P}^r)^{*} \times \mb{P}^r$ consisting of tuples $(F,H_1,p_1,H_2,p_2)$ such that $p_1\neq p_2$, $p_i\in H_i$, and $F$ restricted to $H_i$ vanishes at $p_i$ up to third order.  

To see the projection $\mc{J}\rightarrow \mc{I}$ is not surjective, it suffices to show $\dim(\mc{J})<\dim(\mc{I})$. This type of analysis is also described at the beginning of the proof of Theorem 1.3 in \cite{Eckardt}. They assume characteristic zero throughout the paper, but the assumption on characteristic is not used here. The idea is that the image of $\mc{J}$ in $(\mb{P}^r)^{*}\times \mb{P}^r\times (\mb{P}^r)^{*} \times \mb{P}^r$ decomposes into the following $\mb{P}GL(r+1)$-orbits:
\begin{enumerate}
\item $p_1\notin H_2, p_2\notin H_1$
\item $p_1\in H_2, p_2\notin H_1$ (and similarly the locus obtained by interchanging the indices 1 and 2)
\item $p_1\in H_2, p_2\in H_1$ but $H_1\neq H_2$
\item $H_1=H_2$, $p_1\neq p_2$.
\end{enumerate}
We will do case (3), because it seemed the most worrisome to us. The proofs of the other cases are similar. Without loss of generality, we can assume $p_1=[1:0:\cdots:0]$, $H_1=\{X_1=0\}$, $p_2=[0:0:1:0:\cdots:0]$, $H_2=\{X_3=0\}$. Then, the fiber of $\mc{J}$ over $(p_1,H_1,p_2,H_3)$ consists of the polynomials $\sum_{I}c_I X^I$ such that if
\begin{enumerate}[(a)]
\item
$X_0^{d-2}$ divides $X^I$ but $X_1$ does not or 
\item
$X_2^{d-2}$ divides $X^I$ but $X_3$ does not
\end{enumerate}
then $c_I=0$. Each case gives $\binom{r+1}{2}$ conditions, but there might be overlapping conditions. The number of overlapping conditions is maximized for $d=3$, where it is $r-1$. So the locus of points $(F,p_1,H_1,p_2,H_2)$ in $\mc{J}$ where $(p_1,H_1,p_2,H_2)$ satisfy the conditions of case (3) has dimension
\begin{align*}
\dim(W_{r,d})-\left(2\binom{r+1}{2}-(r-1)\right)+\left(2r+2(r-2)\right). 
\end{align*}
Subtracting this from $\dim(\mc{I})$ yields $\binom{r}{2}-2r+4=\frac{1}{2}(r^2-5r+8)$, which is positive for $r\geq 2$. 
The condition that $r\geq 3$ comes from the part of $\mc{J}$ lying over case (1), and this is clearly necessary as the case $r=2$ corresponds to plane curves and Eckardt points are flex points. 
\end{proof}
\subsection{Completion of proof of Theorem \ref{EVM}}
Now we apply Proposition \ref{UHS} and Lemmas \ref{BP1} and \ref{BP2} to prove Theorem \ref{EVM}. 
\begin{proof}
If we let $\mc{Z}\subset \mc{F}_{r,d}$ be the locus where the fiber of $\Fo(\mc{F}_{r,d}/W_{r,d})\rightarrow F_{r,d}$ has fiber dimension greater than $r-d-1$ and $\pi: \mc{F}_{r,d}\rightarrow W_{r,d}$ be the projection, then $Z=\pi(\mc{Z})\cap U$. The case $d\geq r$ is trivial as $Z$ is precisely the hypersurfaces containing a line, so we only consider when $d<r$. 

Our strategy will be as follows:
\begin{enumerate}
\item 
Use Proposition \ref{UHS} to find the largest component(s) of $\mc{Z}$
\item
Use Lemmas \ref{BP1} and \ref{BP2} to find their  generic fiber dimensions under the map $\mc{Z}\to W_{r,d}$. 
\end{enumerate}

Let $\mc{F}_{r,d}^{\circ}\subset \mc{F}_{r,d}$ denote the open subset of pairs $(F,p)$ where $p\notin V(F,\partial_{X_0}F,\ldots,\partial_{X_r}F)$. If $d<r-1$, then part (1) of Proposition \ref{UHS} shows the unique component of largest dimension of $\mc{C}$ of $\mc{F}_{r,d}^{\circ}\cap \mc{Z}$ consists of pairs $(X,p)$ where $p$ is an Eckardt point of $X$. Part (2) of Lemma \ref{BP2} shows $\mc{C}$ is generically injective onto its image under $\pi$. Part (2) of Lemma \ref{BP1} shows $\pi(\mc{C})\cap U$ is nonempty, so $\pi(\mc{C})\cap U$  is also the unique component of largest dimension of $Z$. 

If $d=r-1$, then part (2) of Proposition \ref{UHS} shows the unique largest and second largest components $\mc{C}_1$ and $\mc{C}_2$ of $\mc{F}_{r,d}^{\circ}\cap \mc{Z}$ in terms of dimensions are respectively the points $(X,p)$ such that $X$ contains a 2-plane containing $p$ and the points $(X,p)$ where $p$ is an Eckardt point of $X$. Furthermore, we can directly compute $\dim(\mc{C}_1)-\dim(\mc{C}_2)=r-3$. As before, $\mc{C}_2$ is generically injective onto its image under $\pi$ and $\pi(\mc{C}_2)\cap U$ is nonempty. Part (1) of Lemma \ref{BP2} shows $\mc{C}_1$ maps onto its image with 2-dimensional fibers and part (1) of Lemma \ref{BP1} shows $\pi(\mc{C}_1)\cap U$ is nonempty for $r\geq 5$. 
\end{proof}

\section{Application: Hypersurfaces singular along a curve}
\label{singh}

We want to show, among the hypersurfaces with positive dimensional singular locus, the unique component of largest dimension consists of the hypersurfaces singular along a line. To prove this in characteristic 0, it will suffice to prove it in characteristic $p$ for one choice of $p$ by an application of upper semicontinuity. We will chose $p=2$ because it gives us the best bounds.

The obstacle to directly applying our general argument to the problem at hand is that the partial derivatives $\partial_{X_i}$ of a degree $\ell$ form $F$ do not vary independently as we vary $F$ in $W_{r,\ell}$. However, the key trick is given in \cite{Poonen} and used in \cite{Kaloyan} to resolve this problem. Let $K$ be characteristic 2 and for simplicity suppose $\ell=2d+1$ is odd. Then, when choosing our degree $\ell$ form $F$, we can add independent fudge factors $G_0,\ldots,G_r$, which are forms of degree $d$, and take the sum
\begin{align*}
F+X_0 G_0^2+\cdots+X_r G_r^2,
\end{align*}
so $\partial_{X_i}(F+X_0 G_0^2+\cdots+X_r G_r^2)=\partial_{X_i}F+G_i^2$. At least optically, it looks like the partial derivatives are more independent, and we will reproduce the same argument Slavov used in \cite{Kaloyan} to reduce the problem of when $F+X_0 G_0^2+\cdots+X_r G_r^2$ is singular along a curve to the problem of when the fudge factors $G_i$ all contain the same curve. As a technical remark, we need to consider the case of hypersurfaces singular along a rational normal curve separately because the bounds given by Theorem \ref{FML} were slightly too weak. Once we remove the locus of all the $G_i$'s containing a rational normal curve, we can repeat the proof of Theorem \ref{FML} to get slightly better bounds that will suffice. 

\subsection{Case of plane curves}
The case $r=2$ of plane curves is easy because everything can be computed explicitly. In the proof of Theorem \ref{oddl} and \ref{evenl}, our bounds will improve with increasing $r$, so it is helpful to be able to assume $r\geq 3$. Also, Claim \ref{ST2} below requires $r\geq 3$. It is an easy dimension computation to see:

\begin{Prop}
\label{psing}
For curves in $\mb{P}^2$, $\dim(\mc{S}_{1,K}^1)>\dim(\mc{S}_{1,K}\backslash \mc{S}_{1,K}^1)$ for all fields $K$ and all degrees $\ell$.
\end{Prop}

\subsection{Reduction to characteristic 2}
\label{SS61}
We will introduce an incidence correspondence over $\Spec(\mb{Z})$ as in \cite[Section 3.1]{Kaloyan}. Fix a degree $\ell\geq 3$ and dimension $r\geq 2$. Let $W_{\mb{Z}}:=\mb{Z}[X_0,\ldots,X_r]_{\ell}$. The notation $W_{r,\ell,\mb{Z}}$ would be more consistent with Definition \ref{WDEF}, but we drop $\ell$ and $r$ from the notation because they are fixed. Over the $\Spec(\mb{Z})$-scheme $W_{\mb{Z}}\cong \mb{A}^{\binom{r+\ell}{\ell}}_{\mb{Z}}$, we can construct a scheme $\mathscr{S} \subset W_{\mb{Z}}\times_{\Spec(\mb{Z})}\mb{P}^r_{\mb{Z}}=\mb{P}^r_{W_{\mb{Z}}}$, where over each point $\Spec(K)\rightarrow W_{\mb{Z}}$ corresponding to a homogenous polynomial $F$ over $K$, the fiber $\mathscr{S}\times_{\mb{P}W_{\mb{Z}}}\Spec(K)$ is the subscheme of $\mathbb{P}^r_K$ cut out by $(F,\partial_{X_0}F,\ldots,\partial_{X_r}F)$. 


By upper semicontinuity of fiber dimension, we can filter 
\begin{align*}
\mb{P}W_{\mb{Z}}=\mc{S}_{-1}\supset \mc{S}_0\supset \mc{S}_1\cdots \supset \mc{S}_i\supset\cdots ,
\end{align*}
 where $\mc{S}_i$ is the closed subset over which the fiber of $\mathscr{S}$ has dimension at least $i$. Put another way, $\mc{S}_i$ is the hypersurfaces that are singular along a subvariety of dimension at least $i$. 

\begin{Def}
Given $\ell$ and $r$ as above, let $\mc{S}_i\subset W_{\mb{Z}}$ be the locus of hypersurfaces that have a singular locus of dimension at least $i$. Let $\mc{S}_{i,K}$ denote $\mc{S}_i\times_{\Spec(\mb{Z})}\Spec(K)$ for a point $\Spec(K)\rightarrow \Spec(\mb{Z})$. 
\end{Def}

\begin{Def}
We let $\mc{S}^1_{i}\subset  \mc{S}_{i}$, where $\mc{S}^1_{i}$ is the locus of hypersurfaces singular along a dimension $i$ plane. As before, we let $\mc{S}_{i,K}^1$ denote the base change of $\mc{S}^1_{i}$ to a field $K$. 
\end{Def}
Recall:
\begin{Thm} [{\cite[Theorem 1.1]{Kaloyan}}]
\label{SlavovT}
Fix $i, r, p$. Then, there is an effectively computable $\ell_0$ in terms of $i,r,p$ such that $\dim(\mc{S}_{i,K}^1)>\dim(\mc{S}_{i,K}\backslash \mc{S}_{i,K}^1)$ for $\ell>\ell_0$ and all fields $K$ of characteristic $p$. 
\end{Thm}

Roughly, the proof of Theorem \ref{SlavovT} bounds $\mc{S}_{i,K}$ by stratifying based on the degree of the variety contained in the singular locus and has a separate argument for the case of low degree and the case of high degree. We will use the argument from the high degree case for hypersurfaces singular along any curve other than a rational normal curve (which includes lines).

In order to keep our statements clean, we will restrict ourselves to the case $i=1$ and the case our base field has characteristic 0. We want to show $\dim(\mc{S}_{i,K}^1)>\dim(\mc{S}_{i,K}\backslash \mc{S}_{i,K}^1)$. Recall:
\begin{Prop}
[{\cite[Lemma 5.1]{Kaloyan}}]
\label{S1d}
We have $\codim(\mc{S}_{1,K}^1)=\ell r - 2 r + 3$ for all fields $K$. 
\end{Prop}

Now, we want to apply upper semicontinuity to show that $\dim(\mc{S}_{1,\overline{\mb{F}_2}}^1)>\dim(\mc{S}_{1,\overline{\mb{F}_2}}\backslash \mc{S}_{1,\overline{\mb{F}_2}}^1)$ implies $\dim(\mc{S}_{1,\overline{\mb{Q}}}^1)>\dim(\mc{S}_{1,\overline{\mb{Q}}}\backslash \mc{S}_{1,\overline{\mb{Q}}}^1)$. In fact, if we just wanted $\dim(\mc{S}_{1,\overline{\mb{Q}}}^1)\geq \dim(\mc{S}_{1,\overline{\mb{Q}}}\backslash \mc{S}_{1,\overline{\mb{Q}}}^1)$, this would follow from upper semicontinuity of fiber dimension applied to $\overline{\mc{S}_{1}\backslash \mc{S}_{1}^1}$. As it stands, we have to worry about the case where a component of $\overline{\mc{S}_{1}\backslash \mc{S}_{1}^1}$ is distinct from $\mc{S}_{1}^1$ over the generic fiber, but limits to $\mc{S}_{1,\overline{\mb{F}_2}}^1$ over the prime $2$. 

\begin{Lem}
\label{RPC}
If $p$ is a prime and $\dim(\mc{S}_{1,\overline{\mb{F}_p}}^1)>\dim(\mc{S}_{1,\overline{\mb{F}_p}}\backslash \mc{S}_{1,\overline{\mb{F}_p}}^1)$, then we also have $\dim(\mc{S}_{1,K}^1)>\dim(\mc{S}_{1,K}\backslash \mc{S}_{1,K}^1)$ for algebraically closed fields of almost all characteristics, including characteristic zero. 
\end{Lem}

\begin{proof}
Consider the incidence correspondence
\begin{center}
\begin{tikzcd}
& \tilde{\mc{S}_1} \arrow[swap, ld, "\pi_1"] \arrow[dr, "\pi_2"] &\\
\mc{S}_1 & & \overline{\RHilb^1_{\mb{P}_{\mb{Z}}^r}}
\end{tikzcd}
\end{center}
where $\Hilb^1_{\mb{P}_{\mb{Z}}^r}$ is the Hilbert scheme of curves in $\mb{P}^r_{\mb{Z}}$, and $\overline{\RHilb^1_{\mb{P}_{\mb{Z}}^r}}$ is the closure of the open sublocus consisting of integral curves. Here, $\tilde{\mc{S}_1}$ consists of pairs $(F,[C])$ of a degree $\ell$ form $F$ and a curve $C$ such that the partial derivatives of $F$ vanish on $C$. More precisely, we can apply Lemma \ref{CFF} to the universal family over $\overline{\RHilb^1_{\mb{P}_{\mb{Z}}^r}}$ and to the family $\mc{S}_1\rightarrow W_{\mb{Z}}$. 

By definition, $\mc{S}^1_1=\pi_1(\pi_2^{-1}(\mb{G}(1,r)))$. Let $\mc{S}^2_1:=\pi_1(\pi_2^{-1}(\overline{\RHilb^1_{\mb{P}_{\mb{Z}}^r}}\backslash\mb{G}(1,r)))$ be the degree $\ell$ forms whose corresponding hypersurfaces are singular along a curve of degree greater than 1. Crucially, $\mc{S}^2_1$ contains, for example, hypersurfaces singular along a scheme supported on a line with multiplicity 2. Since $\mc{S}_1^2$ is closed, $\mc{S}^2_1\supset \overline{\mc{S}_{1}\backslash \mc{S}_{1}^1}$. 

\begin{Claim}
\label{STZ}
For any field $K$, $\mc{S}^2_{1,K}:=\mc{S}^2_1\times_{\Spec(\mb{Z})}\Spec(K)$ does not contain $\mc{S}^1_{1,K}$.
\end{Claim}
First, if we assume Claim \ref{STZ}, then Lemma \ref{RPC} follows from upper semicontinuity applied to $\mc{S}^2_1\rightarrow \Spec(\mb{Z})$ as the fiber dimensions of $\mc{S}^1_1\rightarrow \Spec(\mb{Z})$ are constant by Proposition \ref{S1d} and $\mc{S}^1_{1,K}$ is irreducible for all $K$. To show Claim \ref{STZ}, it suffices to find a single polynomial $F\in K[X_0,\ldots,X_r]$ of degree $\ell$ such that the ideal generated by $(\partial_{X_0} F,\ldots,\partial_{X_r} F)$ scheme theoretically cuts out a curve in $\mb{P}^r$ of degree 1, so Claim \ref{ST1} suffices.

\begin{Claim}
\label{ST1}
Fix a line $L\subset \mb{P}_K^r$ and let $V\subset W_{\mb{Z}}\times_{\Spec(\mb{Z})}\Spec(K)$ be the subvector space of degree $\ell$ forms over $K$ that are singular along $L$. Then, there is a dense open $U\subset V$ consisting of degree $\ell$ forms $F$ where $(F,\partial_{X_0} F,\ldots,\partial_{X_r} F)$ scheme theoretically cut out a curve of degree 1.  
\end{Claim}

To see Claim \ref{ST1}, we first note Lemma 7.4 in \cite{Kaloyan} shows there is a dense open subset $U_1\subset V$ consisting of forms $F$ whose partial derivatives cut out $L$ set-theoretically. (Lemma 7.4 in \cite{Kaloyan} assumes $\ell\geq 3$, though the case $\ell=2$ is also true, for example from the proof of Claim \ref{ST2} below.) We now focus our attention around $L$. Let $\mc{I}_L$ be the ideal sheaf of the line $L$ and $L'\supset L$ be the scheme cut out by $\mc{I}_L^2$. Let $X\rightarrow V$ be the family $X\subset \mb{P}^r\times V$, where each fiber of $X$ over $[F]\in V$ is the scheme cut out in $\mb{P}^r_K$ by the partials  $(F,\partial_{X_0} F,\ldots,\partial_{X_r} F)$. (In the notation at the beginning of section \ref{SS61}, $X$ is $(\mathscr{S}\times_{\Spec(\mb{Z})}\Spec(K))|_{V}$.) 

We consider the intersection $X\cap (L'\times_K V)$ and apply upper semicontinuity of degree to the family $X\cap (L'\times_K V)\rightarrow V$ to see the locus $U_2\subset V$ over which each fiber of $X\cap (L'\times_K V)\rightarrow V$ is degree 1 is open in $V$. To get upper semicontinuity of degree of $X\cap (L'\times_K V)\rightarrow V$, we are using that each fiber is of the same dimension. To prove it in our case, for $p\in V$ and slice $X\cap (L'\times_K V)$ by a general hyperplane $H\subset \mb{P}^r$ such that $X|_p\cap H$ has length equal to $\deg(X|_p)$. Then, since $H$ cannot contain the support of $L$, $X\cap (H\cap L'\times_K V)\rightarrow V$ is a finite morphism, and we can apply upper semicontinuity of rank of a coherent sheaf to the pushforward of the structure sheaf of $X\cap (H\cap L'\times_K V)$ to $V$ to conclude. 

Finally, if we knew $U_2$ were nonempty, then $U_1\cap U_2$ would satisfy the conditions of Claim \ref{ST1}. 
\begin{Claim}
\label{ST2}
The set $U_2\subset V$ is nonempty.
\end{Claim}
Without loss of generality, suppose the ideal sheaf of $L$ is generated by $(X_0,\ldots,X_{r-2})$. If $r\geq 3$ and $Q$ is a degree 2 form in $X_0,\ldots,X_{r-2}$ that cuts out a smooth quadric in $\mb{P}^{r-2}$. Here, we are using the fact that we can assume $K$ is algebraically closed. Since the partial derivatives of $Q$ are linear, the partial derivatives of $Q$ generate $(X_0,\ldots,X_{r-2})$ exactly. If we consider $Q$ as a form in $X_0,\ldots,X_r$ that ignores the last two varibles, the partial derivatives of $Q$ generate exactly the ideal sheaf of a line. 

If $\ell=2$, then we are done. Otherwise, pick general linear forms $H_1,\ldots,H_{\ell-2}$ that all intersect $L$ properly. Then, the product $Q H_1\cdots H_{\ell-2}$ is in  $U_2$. 

Note that in the proof we use that $r>2$, since in the case $r=2$ and characteristic 2, then the singular locus of $V(X_0^2 G)$ for $G$ a degree $\ell-2$ form has $V(X_0^2)$ in the singular locus. Specifically, the proof above fails when $r=2$ since we can't pick a smooth quadric in only one variable. 
\end{proof}

From Theorems \ref{oddl} and \ref{evenl}, we have:

\begin{Thm}
\label{alll}
For $\ell\geq 7$ or $\ell=5$, $\dim(\mc{S}_{1,\overline{\mb{F}_2}}^1)>\dim(\mc{S}_{1,\overline{\mb{F}_2}}\backslash \mc{S}_{1,\overline{\mb{F}_2}}^1)$, so in particular $\dim(\mc{S}_{1,K}^1)>\dim(\mc{S}_{1,K}\backslash \mc{S}_{1,K}^1)$  for algebraically closed fields $K$ of characteristic 0 or of characteristic $p$ for all but finitely many $p$. 
\end{Thm}

\subsection{Counting in characteristic $p$}
We will use a clever trick first given in \cite{Poonen} and then used in \cite{Kaloyan}. To apply this trick, we need the Lang-Weil estimate to relate counting rational points in characteristic $p$ to dimension.  Recall:
\begin{Thm}
[{\cite[Theorem 1]{LW}}]
\label{LWb}
Suppose $Z\subset \mb{P}^r$ is an irreducible projective variety defined over $\mb{F}_q$. Then, if $\# Z(\mb{F}_{q^c})$ is the number of $\mb{F}_{q^c}$ rational points of $Z$,
\begin{align*}
| \# Z(\mb{F}_{q^c})-q^{\dim(Z)}|\leq \delta q^{\dim(Z)-\frac{1}{2}}+A(r,\dim(Z),\deg(Z))q^{r-1},
\end{align*}
where $\delta=(\deg(Z)-1)(\deg(Z)-2)$ and $A$ is a function of $\dim(Z)$, $\deg(Z)$ and $r$. 
\end{Thm}
We rephrase Theorem \ref{LWb} in a weaker form that is easier to apply. 
\begin{Lem}
\label{LWb2}
If $X$ is a quasiprojective variety defined over $\mb{F}_q$ and $Y\subset X$ is a constructible set, then the codimension of $Y$ in $X$ is greater than $A$ if and only if 
\begin{align*}
\lim_{c\rightarrow\infty}q^{cA}\Prob(x\in Y(\mb{F}_{q^c})):=\lim_{c\rightarrow\infty}q^{cA}\frac{\# Y(\mb{F}_{q^c})}{\# X(\mb{F}_{q^c})}=0.
\end{align*}
\end{Lem}

\begin{proof}
Let $X\subset \mb{P}^r$ be a locally closed embedding and $\overline{X}$ be its closure in projective space. Then, $\# X(\mb{F}_{q^c})=\Theta(q^{c\dim(X)})$, as $\dim(\overline{X}\backslash X)<\dim(X)$ and we can apply Theorem \ref{LWb} to every irreducible component of $\overline{X}$ and of $\overline{X}\backslash X$. 

Let $Y^{\circ}\subset Y\subset \overline{Y}$, where $Y^{\circ}$ is a quasiprojective variety contained in $Y$ and $\overline{Y}$ is the closure of $Y$ in $\overline{X}$. By applying Theorem \ref{LWb} to each irreducible component of $\overline{Y}$ and to each irreducible component of $\overline{Y}\backslash Y^{\circ}$, we find $\# Y(\mb{F}_{q^{c}})=\Theta(q^{c\dim(Y)})$.
\end{proof}

See Lemma 6.3 in \cite{Kaloyan} for a slightly more general version of Lemmas \ref{lp1} and \ref{lp2}. 
\begin{Lem}
\label{lp1}
Let $\ell$ be odd. Fix a reduced scheme $Z\subset \mb{P}^r$ defined over $\mb{F}_q$ for $q$ a power of $2$. If we fix $G\in \mb{F}_{q}[X_0,\ldots,X_r]_{\ell-1}$ and pick $G_0\in \mb{F}_{q}[X_0,\ldots,X_r]_{\frac{\ell-1}{2}}$ randomly then
\begin{align*}
\Prob(V(G+G_0^2)\supset Z)\leq \Prob(V(G_0)\supset Z). 
\end{align*}
\end{Lem}

\begin{Lem}
\label{lp2}
Let $\ell$ be even. Fix a reduced scheme $Z\subset \mb{P}^r$ defined over $\mb{F}_q$ for $q$ a power of 2 and with no component contained in the hyperplane $\{X_0=0\}$. If we fix $G\in \mb{F}_{q}[X_0,\ldots,X_r]_{\ell-1}$ and pick $G_0\in \mb{F}_{q}[X_0,\ldots,X_r]_{\frac{\ell}{2}-1}$ randomly then
\begin{align*}
\Prob(V(G+X_0G_0^2)\supset Z)\leq \Prob(V(G_0)\supset Z). 
\end{align*}
\end{Lem}

\begin{proof}
Since the proofs of Lemma \ref{lp1} and Lemma \ref{lp2} are exactly the same, we will prove Lemma \ref{lp2}. We will show equality holds if $\Prob(V(G+X_0 G_0^2)\supset Z)>0$. Suppose $\Prob(V(G+G_0^2)\supset Z)>0$. Fix $G_1\in \mb{F}_{q}[X_0,\ldots,X_r]_{\frac{\ell-1}{2}}$ such that $V(G+X_0G_1^2)\supset Z$. Then, for every other choice $G_0\in \mb{F}_{q}[X_0,\ldots,X_r]_{\frac{\ell}{2}-1}$ such that $V(G+X_0G_0^2)\supset Z$, we see that $X_0G_0^2-X_0G_1^2=X_0(G_0-G_1)^2$ contains $Z$. Since $Z$ reduced and $\{X_0=0\}$ does not contain a component of $Z$, $X_0$ restricts to a nonzero divisor on $Z$ and $V(G_0-G_1)\supset Z$. Thus, the map 
\begin{align*}
\{G_0\in \mb{F}_{q}[X_0,\ldots,X_r]_{\frac{\ell}{2}-1}\mid V(G+X_0 G_0^2)\supset Z\}\rightarrow \{G_0\in \mb{F}_{q}[X_0,\ldots,X_r]_{\frac{\ell}{2}-1}\mid V(G_0)\supset Z\}
\end{align*}
that sends $G_0$ to $G_0-G_1$ is a bijection. Put another way, the first set is a torsor under the action of the second set under addition.
\end{proof}

\subsection{Hypersurfaces singular along a rational normal curve}
We will need to know the number of conditions it is to be singular along a fixed rational normal curve of degree $r$ in $\mb{P}^r$. We will use:
\begin{Lem}
[{\cite[Proposition 8]{Conca}}]
\label{Rs}
If $\ell\geq 3$, $C\cong \mb{P}^1\hookrightarrow \mb{P}^r$ is a fixed rational normal curve of degree $r$,  and $V \subset H^{0}(\mb{P}^r,\mathscr{O}_{\mb{P}^r}(\ell))$ is the vector space of degree $\ell$ forms singular along $C$, then $V$ is of codimension $r^2(\ell+1)-2(r^2-1)$. 
\end{Lem}

\begin{Lem}
\label{Rs2}
If $\ell\geq 3$, $C\cong \mb{P}^1\hookrightarrow \mb{P}^r$ is a fixed rational normal curve of degree $r\geq 3$, $\mb{P}^r\subset \mb{P}^{r+a}$ is embedded as a linear subspace, $V\subset H^{0}(\mb{P}^{r+a},\mathscr{O}_{\mb{P}^{r+a}}(\ell))$ is the vector space of degree $\ell$ forms singular along $C$, then $V$ is of codimension $r^2(\ell+1)-2(r^2-1)+a(r(\ell-1)+1)$. 
\end{Lem}

\begin{proof}
Since $C$ is a local complete intersection, from the introduction of \cite{Symb2}, it suffices to find the Hilbert function of $\mathscr{O}_{\mb{P}^{r+a}}/\mathcal{I}_C^2$. First, we see that $\mathscr{O}_{\mb{P}^{r+a}}/\mathcal{I}_C^2(\ell)$ has no higher cohomology from the exact sequence 
\begin{align*}
0\rightarrow \mc{I}_C/\mc{I}_C^2\rightarrow \mathscr{O}_{\mb{P}^{r}}/\mc{I}_C^2\rightarrow \mathscr{O}_{\mb{P}^{r}}/\mc{I}_C\rightarrow 0.
\end{align*}
Indeed, $H^{1}(\mathscr{O}_{\mb{P}^{r}}/\mc{I}_C^2(\ell))$is 0 as $H^{1}(\mathscr{O}_{\mb{P}^{r}}/\mc{I}_C(\ell))\cong H^{1}(\mathscr{O}_{\mb{P}^1}(r\ell))=0$ and $H^1(\mc{I}_C/\mc{I}_C^2(\ell))=H^1(\mathscr{O}_{\mb{P}^1}(r\ell - r -2 ))^{r-1}\oplus H^1(\mathscr{O}_{\mb{P}^1}(r\ell -1))^a=0$ \cite[Example 3.4]{normal} . 

Let $V'\subset H^{0}(\mb{P}^{r+a},\mathscr{O}_{\mb{P}^{r+a}}(\ell))$ be the subspace of forms vanishing on $C$. There is an induced map $V'\rightarrow H^{0}(\mathcal{I}_C/\mathcal{I}_C^2(\ell))$. Since $C$ is projectively normal, it suffices to show that the map $V'\rightarrow H^{0}(\mathcal{I}_C/\mathcal{I}_C^2(\ell))=H^{0}(N_{C/\mathbb{P}^r}^{\vee}(\ell))\oplus H^{0}(\mathscr{O}_C(\ell-1))^a$ is surjective. 

Given a form $F(X_0,\ldots,X_{r+a})$ vanishing on $C$, we can write it as 
\begin{align*}
F=G(X_0,\ldots,X_r)+G_1(X_0,\ldots,X_{r+a})X_{r+1}+\cdots G_a(X_0,\ldots,X_{r+a})X_{r+a}. 
\end{align*}
The map $V'\rightarrow H^{0}(\mathcal{I}_C/\mathcal{I}_C^2(\ell))=H^{0}(N_{C/\mathbb{P}^r}^{\vee}(\ell))\oplus H^{0}(\mathscr{O}_C(\ell-1))^a$ sends $F$ to 
\begin{align*}
(G|_{N_{C/\mathbb{P}^r}^{\vee}(\ell)},G_1|_{\mathscr{O}_C(\ell-1)},\ldots,G_a|_{\mathscr{O}_C(\ell-1)}). 
\end{align*}
Since the map $V'\rightarrow H^{0}(N_{C/\mathbb{P}^r}^{\vee}(\ell))\cong H^{0}(\mathscr{O}_{\mb{P}^1}(-r-2+\ell r))^{r-1}$ is surjective from Lemma \ref{Rs2}, and the $G_i$'s can be chosen independently, we see that 
\begin{align*}
V'\rightarrow H^{0}(\mathcal{I}_C/\mathcal{I}_C^2(\ell))=H^{0}(N_{C/\mathbb{P}^r}^{\vee}(\ell))\oplus H^{0}(\mathscr{O}_C(\ell-1))^a
\end{align*}
 is surjective.
\end{proof}

\begin{Lem}
\label{srnci}
Let $T_i'\subset H^{0}(\mb{P}^{r},\mathscr{O}_{\mb{P}^{r}}(\ell))$ be the locus of hypersurfaces singular along some degree $i$ rational normal curve. Then, $\codim(T_i')\geq i (\ell r - 2 r - 2) + 5$ and equality holds when $i=1$. 
\end{Lem}

\begin{proof}
From Lemma \ref{Rs2}, the number of conditions it is to be singular along a rational normal curve of degree $i$ is 
\begin{align*}
i^2(\ell+1)-2(i^2-1)+(r-i)(i(\ell-1)+1).
\end{align*}
The space of rational normal curves of degree $i$ in $\mb{P}^r$ is dimension $(i+3)(i-1)+(i+1)(r-i)$, so 
\begin{align*}
\codim(T_i')&\geq i^2(\ell+1)-2(i^2-1)+(r-i)(i(\ell-1)+1)-(i+3)(i-1)-(i+1)(r-i)\\
\codim(T_i')&\geq i (\ell r - 2 r - 2) + 5.
\end{align*}
Equality holds when $i=1$ from Lemma 5.1 in \cite{Kaloyan}. 
\end{proof}

\subsection{Hypersurfaces containing a nondegenerate curve}
We need to repeat the proofs for Lemmas \ref{KIA} and Lemma \ref{ICA} but we want to remove rational normal curves from our set of curves.

\begin{Def}
Let $\Span(b)'\subset \RHilb_{\mb{P}^r}^1$ denote the integral curves whose span is exactly a $b$-dimensional plane except for the rational normal curves of degree $b$.
\end{Def}

From a special case of \cite[Theorem 4.5]{Park}, we have
\begin{Lem}
\label{ABP}
We have
\begin{align*}
h_{\Span(b)'\cap \RHilb^1_{\mb{P}^r}}(d)\geq (b+1)d.
\end{align*}
\end{Lem}

From iterating Lemma \ref{kind} together with applying \ref{ABP}, we find
\begin{Lem}
\label{KIAP}
For $r-k+a=1$, we have $\codim(\Phi^{\mb{P}^r,a}_{d_1,\ldots,d_k}(\mb{P}^r,\Span(r)'))\geq a(r+1)d$ for $d_1=\cdots=d_k=d$.
\end{Lem}
\begin{proof}
It suffices to show $h_{\Span(r)'\cap \RHilb^1_{\mb{P}^r}}(d)\leq h_{\Span(r)\cap \RHilb_{\mb{P}^r}^2}(d)$ for all $d$. Equivalently,
\begin{align*}
(r-2)\binom{d+1}{2}+\binom{d+2}{2}&\geq (r+1)d=(r-1)\binom{d}{1}+\binom{d+1}{1}+(d-1)\\
(r-1)\binom{d}{2}+\binom{d+1}{2}-\binom{d+1}{2}&\geq d-1\\
(r-1)\frac{d}{2}(d-1) &\geq d-1,
\end{align*}
which is true when $d\geq 2$. 
\end{proof}

Similarly, applying Lemma \ref{IC} and following the same proof as Lemma \ref{ICA}, we get
\begin{Lem}
\label{ICAP}
For $r-k+a=1$,
\begin{align*}
\codim(\Phi^{\mb{P}^r,a}_{d_1,\ldots,d_k}(\mb{P}^r,\Span(b)'))&\geq \codim(\Phi^{\mb{P}^b,a+(r-b)}_{d_1,\ldots,d_k}(\mb{P}^b,\Span(b)'))-\dim(\mb{G}(b,r)).
\end{align*}
\end{Lem}

Combining Lemmas \ref{KIAP} and Lemma \ref{ICAP}, we get
\begin{Lem}
\label{klpc}
For $r-k+a=1$, we have 
\begin{align*}
\codim(\Phi^{\mb{P}^r,a}_{d_1,\ldots,d_k}(\mb{P}^r,\Span(b)'))\geq (a+r-b)(b+1)d-\dim(\mb{G}(b,r))
\end{align*} for $d_1=\cdots=d_k=d$.
\end{Lem}

\subsection{Case of odd degree}
The argument is the same in both cases, but it is cleaner in the case where $\ell$ is odd, giving slightly better bounds.

\begin{Thm}
\label{oddl}
For odd $\ell\geq 5$, $\dim(\mc{S}_{1,\overline{\mb{F}_2}}^1)>\dim(\mc{S}_{1,\overline{\mb{F}_2}}\backslash \mc{S}_{1,\overline{\mb{F}_2}}^1)$.
\end{Thm}

\begin{proof}
From Proposition \ref{psing}, we can assume $r\geq 3$. We cover $\mc{S}_{1,\overline{\mb{F}_2}}$ by constructible sets:
\begin{enumerate}
\item
$T_1=\mc{S}_{1,\overline{\mb{F}_2}}^1$
\item
$T_i\subset \mc{S}_{1,\overline{\mb{F}_2}}$ be the hypersurfaces singular along a degree $i$ rational normal curve
\item
$T_i'\subset \mc{S}_{1,\overline{\mb{F}_2}}$ be the hypersurfaces singular along some integral curve that
\begin{enumerate}[(i)]
\item
 spans an $i$-dimensional plane 
 \item 
 is not a degree $i$ rational normal curve.
 \end{enumerate}
\end{enumerate}
We will bound the codimensions of the sets $T_i$, $T_i'$ for $2\leq i\leq r$ separately. 

Let $A=\codim(T_1)=\ell r - 2 r + 3$. First, $\codim(T_i)>\codim(T_1)$ for $i>1$, from Lemma \ref{srnci} as $r,\ell\geq 3$ implies $\ell r - 2 r - 2>0$.  

We bound $\codim(T_i')$ for $2\leq i\leq r$. From Lemma \ref{LWb2}, it suffices to show
\begin{align*}
\lim_{c\rightarrow\infty}2^{cA}\Prob(F\in T_i')=0,
\end{align*}
where for each $c$ we select a hypersurface $F$ randomly from $\mb{F}_{2^c}[X_0,\ldots,X_r]_\ell$. Note that selecting $F$ randomly from $\mb{F}_{2^c}[X_0,\ldots,X_r]_\ell$ is equivalent to picking $(G,G_0,\ldots,G_r)$ from $\mb{F}_{2^c}[X_0,\ldots,X_r]_\ell\times (\mb{F}_{2^c}[X_0,\ldots,X_r]_{\frac{\ell-1}{2}})^{r+1}$ randomly and letting $F=G+X_0G_0^2+\cdots+X_r G_r^2$. 

Then, $\partial_iF = \partial_i G + G_i^2$. Let $E_i$ be the condition that $ \{\partial_m G + G_m^2\mid 0\leq m\leq r \}$ all vanish on some integral curve spanning an $i$-dimensional plane other than a degree $i$ rational normal curve. It suffices to show
\begin{align*}
\lim_{c\rightarrow\infty}2^{cA}\Prob(E_i)=0. 
\end{align*}
Let $E_i'$ be the condition that $G_0,\ldots,G_r$ all contain an integral curve spanning an $i$-dimensional plane other than a degree $i$ rational normal curve. From Lemma \ref{lp1}, it suffices to show
\begin{align*}
\lim_{c\rightarrow\infty}2^{cA}\Prob(E_i')=0. 
\end{align*}
Applying Lemma \ref{LWb2} again, it suffices to show 
\begin{align*}
\codim(\Phi^{\mb{P}^r,2}_{d_1,\ldots,d_{r+1}}(\mb{P}^r,\Span(i)'))>A
\end{align*}
for $d_1=\cdots=d_{r+1}=d$. 
Applying Lemma \ref{klpc}, we see
\begin{align*}
\codim(\Phi^{\mb{P}^r,2}_{d_1,\ldots,d_{r+1}}(\mb{P}^r,\Span(i)'))&\geq -(i+1)(r-i)+(r-i+2)((i+1)d)\\
\codim(\Phi^{\mb{P}^r,2}_{d_1,\ldots,d_{r+1}}(\mb{P}^r,\Span(i)'))&\geq (1 - d) i^2 + i (d r + d - r + 1) + d r + 2 d - r
\end{align*}
for $d_1=\cdots=d_{r+1}=d$. Since this is quadratic in $i$ with negative leading coefficient, it suffices to check the cases $i=2$ and $i=r$. We see for $i=2$, 
\begin{align*}
\codim(\Phi^{\mb{P}^r,2}_{d_1,\ldots,d_{r+1}}(\mb{P}^r,\Span(2)'))&\geq -3(r-2)+r(3d)\\
\codim(\Phi^{\mb{P}^r,2}_{d_1,\ldots,d_{r+1}}(\mb{P}^r,\Span(2)'))-A&\geq d r - 2 r + 3,
\end{align*}
so $d\geq 2$ suffices. In the case $i=r$,
\begin{align*}
\codim(\Phi^{\mb{P}^r,2}_{d_1,\ldots,d_{r+1}}(\mb{P}^r,\Span(r)'))&\geq 2(r+1)d\\
\codim(\Phi^{\mb{P}^r,2}_{d_1,\ldots,d_{r+1}}(\mb{P}^r,\Span(r)'))-A&\geq r+2 d - 3,
\end{align*}
so $d\geq 2$ suffices. 
\end{proof}

\subsection{Case of even degree}
\begin{Thm}
\label{evenl}
For even $\ell\geq 8$, $\dim(\mc{S}_{1,\overline{\mb{F}_2}}^1)>\dim(\mc{S}_{1,\overline{\mb{F}_2}}\backslash \mc{S}_{1,\overline{\mb{F}_2}}^1)$.
\end{Thm}

\begin{proof}
From Proposition \ref{psing}, we can assume $r\geq 3$. Since 
$$\bigcap_{j\neq j'}{\{X_jX_{j'}=0\}}=\{[1:0:\cdots:0],[0:1:\cdots:0],\ldots,[0:0:\cdots:1]\}$$
is finite, it cannot contain any curves. We thus cover $\mc{S}_{1,\overline{\mb{F}_2}}$ by constructible sets:
\begin{enumerate}
\item
$T_1=\mc{S}_{1,\overline{\mb{F}_2}}^1$
\item
$T_i\subset \mc{S}_{1,\overline{\mb{F}_2}}$ be the hypersurfaces singular along a degree $i$ rational normal curve
\item
$T_{i,j,j'}\subset \mc{S}_{1,\overline{\mb{F}_2}}$ be the hypersurfaces singular along some integral curve that
\begin{enumerate}[(i)]
\item spans an $i$ dimensional plane 
\item is not a degree $i$ rational normal curve
\item is not contained in $\{X_jX_{j'}=0\}$.
 \end{enumerate}
\end{enumerate}
Let $A=\codim(T_1)=\ell r - 2 r + 3$. 

First, $\codim(T_i)>\codim(T_1)$ for $i>1$, from Lemma \ref{srnci} as $r,\ell\geq 3$ implies $\ell r - 2 r - 2>0$.  

We bound $\codim(T_{i,j,j'})$ for $2\leq i\leq r$. By symmetry, it suffices to assume $j=0, j'=1$. Let $d=\frac{\ell}{2}-1$. From Lemma \ref{LWb2}, it suffices to show
\begin{align*}
\lim_{c\rightarrow\infty}2^{cA}\Prob(F\in T_{i,j,j'})=0,
\end{align*}
were $F$ is selected randomly from $\mb{F}_{2^c}[X_0,\ldots,X_r]_\ell$ for each $c$. 
Note that we can achieve the same, uniform distribution if we pick $(G,G_0,\ldots,G_r)$ from $\mb{F}_{2^c}[X_0,\ldots,X_r]_\ell\times (\mb{F}_{2^c}[X_0,\ldots,X_r]_{d})^{r+1}$ randomly and let $F=G+X_0X_1 G_0^2 + X_0(X_1G_1^2+\cdots+X_r G_r^2)$. 

Then, $\partial_iF = \partial_i G + X_0 G_i^2$ for $i\neq 0$ and $\partial_0 F = \partial_i G +  X_1 G_0^2$. Let $E_{i,0,1}$ be the condition that $ \{\partial_mF\mid 0\leq m\leq r \}$ all vanish along some integral curve spanning an $i$-dimensional plane, is not contained in $\{X_jX_{j'}=0\}$, and is not a rational normal curve of degree $i$.

It suffices to show
\begin{align*}
\lim_{c\rightarrow\infty}2^{cA}\Prob(E_{i,0,1})=0. 
\end{align*}
Suppose we choose $G_0,\ldots,G_r$ from $(\mb{F}_{2^c}[X_0,\ldots,X_r]_{d})^{r+1}$ uniformly. Let $E_i'$ be the condition that they all contain some integral curve $C$ that spans an $i$-dimensional plane but is not a degree $i$ rational normal curve. By applying Lemma \ref{lp2}, it suffices to show that 
\begin{align*}
\lim_{c\rightarrow\infty}2^{cA}\Prob(E_i')=0. 
\end{align*}
Applying Lemma \ref{LWb2} again, we see it suffices to show $\codim(\Phi^{\mb{P}^r,2}_{d_1,\ldots,d_{r+1}}(\mb{P}^r,\Span(i)'))>A$. Applying Lemma \ref{klpc}, we see
\begin{align*}
\codim(\Phi^{\mb{P}^r,2}_{d_1,\ldots,d_{r+1}}(\mb{P}^r,\Span(i)'))&\geq -(i+1)(r-i)+(r-i+2)((i+1)d)\\
\codim(\Phi^{\mb{P}^r,2}_{d_1,\ldots,d_{r+1}}(\mb{P}^r,\Span(i)'))&\geq (1 - d) i^2 + i (d r + d - r + 1) + d r + 2 d - r
\end{align*}
for $d_1=\cdots=d_{r+1}=d$. Since this is quadratic in $i$ with negative leading coefficient, it suffices to check the cases $i=2$ and $i=r$. We see for $i=2$, 
\begin{align*}
\codim(\Phi^{\mb{P}^r,2}_{d_1,\ldots,d_{r+1}}(\mb{P}^r,\Span(2)'))&\geq -3(r-2)+r(3d)\\
\codim(\Phi^{\mb{P}^r,2}_{d_1,\ldots,d_{r+1}}(\mb{P}^r,\Span(2)'))-A&\geq d r - 3 r + 3,
\end{align*}
so $d\geq 3$ suffices. In the case $i=r$,
\begin{align*}
\codim(\Phi^{\mb{P}^r,2}_{d_1,\ldots,d_{r+1}}(\mb{P}^r,\Span(r)'))&\geq 2(r+1)d\\
\codim(\Phi^{\mb{P}^r,2}_{d_1,\ldots,d_{r+1}}(\mb{P}^r,\Span(r)'))-A&\geq 2 d - 3,
\end{align*}
so $d\geq 2$ suffices. 

Putting everything together, we see that we need $d\geq 3$, so $\ell\geq 8$ suffices. 
\end{proof}

\bibliographystyle{plain}
\bibliography{references.bib}
\end{document}